\documentclass[11pt,a4paper,leqno]{amsart}

\usepackage[latin1]{inputenc}
\usepackage[T1]{fontenc}
\usepackage{amsfonts}
\usepackage{amsmath}
\usepackage{amssymb}
\usepackage{eurosym}
\usepackage{mathrsfs}
\usepackage{palatino}
\usepackage{color}
\usepackage{esint}
\usepackage{url}
\usepackage{verbatim}
\usepackage{enumerate}
\usepackage[dvipsnames]{xcolor}

\usepackage{caption,float}
\usepackage{placeins}
\usepackage{afterpage}

\usepackage{subfig}

\usepackage[pagebackref,hypertexnames=false, colorlinks, citecolor=blue, linkcolor=blue, urlcolor=blue]{hyperref}
\usepackage{amsrefs}

\newcommand{\R}{\mathbb{R}}
\newcommand{\C}{\mathbb{C}}

\newcommand{\N}{\mathbb{N}}

\newcommand{\Z}{\mathbb{Z}}
\newcommand{\E}{\mathbb{E}}

\newcommand{\diam}{\operatorname{diam}}

\newcommand{\CZO}{\textup{CZO}}

\newcommand{\CZ}{\textup{CZ}}

\numberwithin{equation}{section}

\newcommand{\ud}[0]{\,\mathrm{d}}

\newcommand{\esssup}[0]{\operatornamewithlimits{ess\,sup}}

\newcommand{\dist}[0]{\operatorname{dist}}

\newcommand{\abs}[1]{|#1|}
\newcommand{\babs}[1]{\big|#1\big|}
\newcommand{\Babs}[1]{\Big|#1\Big|}

\newcommand{\Norm}[2]{\|#1\|_{#2}}
\newcommand{\bNorm}[2]{\big\|#1\big\|_{#2}}
\newcommand{\BNorm}[2]{\Big\|#1\Big\|_{#2}}

\newcommand{\ave}[1]{\langle #1\rangle}

\newcommand{\bave}[1]{\big\langle #1\big\rangle}

\newcommand{\osc}[0]{\operatorname{osc}}

\newcommand{\BMO}[0]{\operatorname{BMO}}
\newcommand{\bmo}[0]{\operatorname{bmo}}

\newcommand{\supp}[0]{\operatorname{spt}}
\newcommand{\loc}[0]{\operatorname{loc}}

\newcommand{\eps}[0]{\varepsilon}

\newcommand{\ch}[0]{\operatorname{ch}}

\newcommand{\calD}[0]{\mathcal{D}}

\newcommand{\wt}[1]{{\widetilde{#1}}}

\swapnumbers
\theoremstyle{plain}
\newtheorem{thm}[equation]{Theorem}
\newtheorem{lem}[equation]{Lemma}
\newtheorem{prop}[equation]{Proposition}

\theoremstyle{definition}
\newtheorem{defn}[equation]{Definition}

\theoremstyle{remark}

\newtheorem{rem}[equation]{Remark}

\title{Off-diagonal estimates for bi-parameter commutators}

\author{Tuomas Oikari}
\address[T.O.]{Department of Mathematics and Statistics, University of Helsinki, P.O.B. 68, FI-00014 University of Helsinki, Finland}
\email{tuomas.v.oikari@helsinki.fi}

\makeatletter
\@namedef{subjclassname@2010}{%
  \textup{2010} Mathematics Subject Classification}
\makeatother

\thanks{T. Oikari was supported by the Academy of Finland project No. 306901, by the Finnish Centre of Excellence in Analysis and Dynamics Research project No. 307333, and by the three-year research grant of the University of Helsinki No. 75160010.
}
\begin{document}
	
 \begin{abstract}
 	We study the boundedness of commutators of bi-parameter singular integrals between mixed spaces
 	$$
 	[b,T]: L^{p_1}L^{p_2} \to L^{q_1}L^{q_2}
 	$$
 	in the off-diagonal situation $q_i,p_i\in(1,\infty)$ where we also allow $q_i\not= p_i.$
 	Boundedness is fully characterized for several arrangements of the integrability exponents with some open problems presented.
 \end{abstract}
 
\maketitle
\section{Introduction and preliminaries}
The first commutator results concern the commutator of the Hilbert transform
\[
[b,H]f = bHf-H(bf)
\] whose boundedness 
was first characterized in the classical theorem of Nehari in \cite{Nehari1957} through Hankel operators. Later,
 Coifman, Rochberg and Weiss \cite{CRW} generalized Nehari's result and showed that
\begin{equation}\label{eq:commutatorbb}
\|b\|_{\BMO} \lesssim \sum_{i=j}^d\|[b,R_j]\|_{L^p(\R^d) \to L^p(\R^d)} \lesssim \|b\|_{\BMO} :=\sup_I \fint_I  \abs{b-\ave{b}_I}, \qquad p \in (1,\infty),
\end{equation}
where the supremum is taken over all cubes $I \subset \R^d$ and $\ave{b}_I = \frac{1}{|I|} \int_I b$. The upper bound in \eqref{eq:commutatorbb} was proved for a wide class of bounded singular integrals, while the lower bound especially involves the Riesz transforms. Later, the lower bound in \eqref{eq:commutatorbb} was improved separately by both Janson \cite{Janson1978} and Uchiyama \cite{Uchiyama1978} by bringing in certain non-degeneracy and assumptions on the kernel of $T$, especially, their results cover the lower bound \eqref{eq:commutatorbb} with any single Riesz transform (in contrast to \eqref{eq:commutatorbb} involving all the $d$ Riesz transforms).  Janson \cite{Janson1978} also covers the off-diagonal situation when $1<p < q<\infty$ and provides the characterization
\allowdisplaybreaks
\begin{equation*}
	\Norm{[b,T]}{L^p\to L^q}\sim \sup_Q\ell(Q)^{-\alpha}\fint_Q\abs{b-\ave{b}_Q}, \qquad \alpha :=d\Big(\frac{1}{p}-\frac{1}{q}\Big).
\end{equation*}
The remaining range with $1<q<p<\infty$ was characterised recently by Hyt\"onen \cite{HyLpLq},
$$
\|[b,T]\|_{L^p \to L^q} \sim  \inf_{c \in \C} \Norm{b-c}{L^r}, \qquad \frac{1}{q}  =\frac{1}{r}+\frac{1}{p}.
$$

The $p=q$ characterization yields factorizations of $H^1$, see \cite{CRW}, and implies div-curl lemmas relevant for compensated compactness, see \cite{CLMS1993}. The sub-diagonal case $q>p$ also implies factorization results, this time for $H^s,$ where $s<1$ now depends on $p,q,$ see for example \cite{kuffner2018weak}.  In Lindberg \cite{Lindberg2017} and Hyt\"{o}nen \cite{HyLpLq} the characterization of the case $q<p$ is connected with a conjecture of Iwaniec \cite{Iwa1997} on the prescribed Jacobian problem. It is crucial in all of these applications that we have both commutator upper and lower bounds.

In this paper, we work in the product ambient space  $\R^d = \R^{d_1}\times \R^{d_2}$ and study the boundedness of the bi-parameter commutators $[b,T],$ where $T$ is now a bi-parameter singular integral operator. Due to the product space nature of the problem, it is natural to allow different integrability exponents in the first and the second parameter, thereby, leading to the question of $L^{p_1}L^{p_2}$-to-$L^{q_1}L^{q_2}$ boundedness. In accordance with the three qualitatively different regimes $p<q$, $p=q$ and $p>q$ in the one-parameter setup, there will now be nine cases depending on the relative size of both of the pairs $p_1,q_1$  and $p_2,q_2.$ 
The exact statements of our results are spread throughout the text; the following Theorem \ref{thm:main} is a condensed version of the obtained results.
\begin{thm}\label{thm:main}
	Let $T$ be a non-degenerate bi-parameter Calder\'on-Zygmund operator on $\R^d = \R^{d_1}\times \R^{d_2},$
	fix the exponents $p_1, p_2, q_1, q_2 \in (1,\infty)$ and set	
	\allowdisplaybreaks \begin{align*}
	\begin{split}
	\alpha_i :=d_i\Big(\frac{1}{p_i}-\frac{1}{q_i}\Big),\quad \text{if}\quad p_i<q_i;\qquad
	\frac{1}{r_i} :=\frac{1}{q_i}-\frac{1}{p_i},\quad \text{if}\quad p_i>q_i.
	\end{split}
	\end{align*}
Let also $b:\R^d\to\C$ be a function with some local integrability depending on $p_1,p_2,q_1,q_2$ ($L^{\infty}_{\loc}$ works in all cases, for example). 
	Then, denoting  $\|[b,T]\|_{L^{p_1}_{x_1}L^{p_2}_{x_2}\to L^{q_1}_{x_1}L^{q_2}_{x_2}} = N_{p,q}$ we have the upper- and lower bounds
	\begin{center}
		\begin{tabular}{ c | c | c | c }
			& $p_1<q_1$ & $p_1=q_1$ & $p_1>q_1$ \\ \hline
			& &  & \\
			
			$p_2<q_2$ &  $\mbox{b = constant,}$	&  $b(\cdot,x_2) = \mbox{constant},$	
			& $\mbox{b = constant},$ \\
			& $N_{p,q} = 0$ & $N_{p,q}\sim \Norm{b(x_1,\cdot)}{\dot C^{0,\alpha_2}_{x_2}}$ & $N_{p,q}=0$ \\
			\hline
			& &  & \\
			&$b(x_1,\cdot) = \mbox{constant},$ & $N_{p,q}\sim \|b\|_{\bmo(\R^{d_1}\times \R^{d_2})}$  &   $\inf_{c\in\C}\Norm{b-c}{L^{\infty}_{x_2} L^{r_1}_{x_1}}\lesssim N_{p,q}$ \\
			$p_2=q_2$ &  $N_{p,q}\sim \Norm{b(\cdot,x_2)}{\dot C^{0,\alpha_1}_{x_1}}$	&& $ \lesssim \inf_{c\in\C}\Norm{b-c}{ L^{r_1}_{x_1}L^{\infty}_{x_2}}$ \\
			\hline
			& &  & \\
			$p_2>q_2$ & $\mbox{b = constant}$	& $N_{p,q}\sim \inf_{c\in\C}\Norm{b-c}{L^{\infty}_{x_1}L^{r_2}_{x_2}}$ & $N_{p,q}  \lesssim \inf_{c\in\C}\| b-c\|_{ L^{r_1}_{x_1} L^{r_2}_{x_2}} $	\\
			& $N_{p,q} = 0$ & & \\
		\end{tabular}
	\end{center}
\end{thm}

Our main focus is on the off-diagonal cases $(p_1, p_2) \ne (q_1, q_2)$ with the diagonal being well-known and lately studied e.g. by Holmes, Petermichl and Wick \cite{HPW2018}, and by Li, Martikainen and Vuorinen \cite{LMV2019biparcom}.

While some of the upper bounds in the off-diagonal situation in the table of Theorem \ref{thm:main} are quick by few applications of Hölder's inequality, or trivial in the constant cases, the rest are not completely effortless and require e.g. the use of representation theorem and other purely bi-parameter tools, however, the most work is found with the lower bounds. We prove the lower bounds through the approximate weak factorization argument but now in the bi-parameter setting. In the two cases where we fail to achieve a full characterization, the problems are mainly due to the fact that the awf argument is symmetric with respect to both of the parameters, while the norm $\Norm{\cdot}{L^{t}_{x_1}L^{s}_{x_2}}$ has a built-in order to it. This limitation is not new and was expected, as we already saw it in Airta, Hyt\"{o}nen, Li, Martikainen and Oikari \cite{AHLMO}, where we provided a similar table as in Theorem \ref{thm:main} above, but for the bi-parameter commutator $[T_2,[b,T_1]],$ where each $T_i$ is a singular integral on $\R^{d_i}.$ In \cite{AHLMO} we achieved a fully satisfactory characterization of the boundedness of the commutator in only four cases, this is in line with $[T_2,[b,T_1]]$ being considered a harder operator to work with than $[b,T].$ Perhaps this difference is best reflected through the fact that the diagonal characterization in terms of the proposed product BMO is open in the first case, see e.g. the discussion in \cite{AHLMO}, whereas the boundedness of $[b,T]$ on the diagonal is fully understood and captured by the simpler little bmo.

\subsection*{Acknowledgements}
We thank Henri Martikainen, Emil Vuorinen and Tuomas Hyt\"{o}nen for their comments that improved the paper.

\vspace{0.2cm}
In the remaining part of this section we provide the definition of singular integrals and commutators. The reader who is familiar with this material may immediately skip the to next Section \ref{sect:awf}.
\subsection{Singular integrals}
We denote the diagonal with $\Delta = \Delta^{(d_i,d_i)} =  \{(x_i,y_i) \in \R^{d_i} \times \R^{d_i} \colon x_i =y_i \}$ and call $$K_i \colon \R^{d_i} \times \R^{d_i} \setminus \Delta \to \C$$ a standard Calder\'on-Zygmund kernel on $\R^{d_i}$
if the size estimate
\begin{align*}
	|K_i(x_i,y_i)| \le \frac{C}{|x_i-y_i|^{d_i}},
\end{align*}
and, for some $\delta > 0$, the regularity estimates
\begin{equation*}\label{eq:holcon}
|K_i(x_i, y_i) - K_i(x_i', y_i)| + |K_i(y_i, x_i) - K_i(y_i, x_i')| \le C\frac{|x_i-x_i'|^{\delta}}{|x_i-y_i|^{d_i+\delta}}
\end{equation*}
whenever $|x_i-x_i'| \le |x_i-y_i|/2$, are satisfied. 
The best constant in these estimates is denoted by $\|K\|_{\textup{CZ}(d_i,\delta)}$ and the collection of all such kernels is denoted as $\CZ(d_i,\delta).$

\begin{defn} Let 
	$\Sigma_i =\Sigma(\R^{d_i})$ be the linear span of the indicator functions of cubes. A singular integral operator (SIO) is then a linear mapping $T_i:\Sigma_i\to L^1_{\loc}(\R^{d_i})$ such that 
	\[
	\langle T_if,g\rangle = \int_{\R^{d_i}}\int_{\R^{d_i}}K(x,y)f(y)g(x)\ud y \ud x,\qquad\supp(f)\cap\supp(g) = \emptyset,\qquad f,g\in\Sigma_i, 
	\]
	where $K\in\CZ(d_i,\delta).$
\end{defn}

\begin{defn} A Calder\'on-Zygmund operator (CZO) is simply an SIO $T_i$ that is bounded
	from $L^p(\R^{d_i}) \to L^p(\R^{d_i})$ for all (equivalently, for some) $p \in (1,\infty)$. Given a CZO $T_i$ with a kernel $K_i\in \CZ(d_i,\delta),$ let us denote
	$\Norm{T}{\CZO(d_i,\delta)} = \Norm{T}{L^2(\R^{d_i})\to L^2(\R^{d_i})} + \Norm{K_i}{\CZ(d_i,\delta)}.$
	
\end{defn}

\subsection{Bi-parameter singular integrals}\label{section:biparSIOs} We give the definition of Martikainen \cite{Ma1} of bi-parameter SIOs, see also the last Section \ref{sect:last} for the original definition by Journée. Now we start working in the ambient space $\R^d = \R^{d_1}\times \R^{d_2}.$ Again, we let 
$\Sigma_i =\Sigma(\R^{d_i})$ be the linear span of the indicator functions of the cubes of $\R^{d_i}$ and then let $\Sigma = \Sigma(\R^{d})$ be the linear span of  $\Sigma_1\otimes \Sigma_2 = \{f_1\otimes f_2 : f_i\in\Sigma_i\}.$ We assume that we are given a linear operator $T$ along with a full adjoint $T^*$ and partial adjoints $T^{*}_1,T^{*}_2,$ i.e., four operators $T,T^*,T^{*}_1,T^{*}_2:\Sigma\to L^1_{\loc}(\R^d)$
that satisfy
\allowdisplaybreaks \begin{align*}
	\langle T(f_1\otimes f_2),  g_1\otimes g_2  \rangle& = 	\langle T^*_1(g_1\otimes f_2),  g_1\otimes g_2  \rangle \\ 
	&= 	\langle T^*_2(f_1\otimes g_2),  g_1\otimes f_2  \rangle =	\langle T^*(g_1\otimes g_2),  f_1\otimes f_2  \rangle.
\end{align*}
These operators will be assumed to have bi-parameter kernels, recalled next.
\subsubsection{Bi-parameter kernels}
Let $\delta>0.$ We assume to have a kernel
\[
K:\R^d\times\R^d \setminus\Delta\to \C,
\]
where $\Delta= \{ (x,y)\in(\R^{d_1}\times\R^{d_2})^2 : x_1= y_1\mbox{ or } x_2 = y_2\},$
that satisfies the size estimate
\allowdisplaybreaks \begin{align}\label{ker:size}
\abs{K(x,y)} \leq C\abs{x_1-y_1}^{-d_1}\abs{x_2-y_2}^{-d_2},
\end{align}
the regularity estimate
\allowdisplaybreaks \begin{align*}
\abs{K(x,y)-& K((x_1,x_2'),y) - K((x_1',x_2),y) + K(x',y)} \\
&\leq C \frac{\abs{x_1-x_1'}^{\delta}}{\abs{x_1-y_1}^{d_1+\delta}}\frac{\abs{x_2-x_2'}^{\delta}}{\abs{x_2-y_2}^{d_2+\delta}},
\end{align*}
whenever $\abs{x_i-x_i'}\leq \frac{1}{2}\abs{x_i-y_i}$ for $i=1,2,$
and the mixed size-regularity estimate
\allowdisplaybreaks \begin{align*}
\abs{K((x_1,x_2),y) - K((x_1',x_2),y)} \leq C \frac{\abs{x_1-x_1'}^{\delta}}{\abs{x_1-y_1}^{d_1+\delta}} \abs{x_2-y_2}^{-d_2},
\end{align*}
whenever $\abs{x_1-x_1'}\leq \frac{1}{2}\abs{x_1-y_1}.$ We also assume the symmetric estimates to the stated regularity and size-regularity estimates to hold in the other parameter slots. The collection of all such kernels is denoted $\CZ((d_1,d_2),\delta)$ and the best constant $C$ in these estimates is denoted with $\Norm{K}{\CZ((d_1,d_2),\delta)}.$ 

\subsubsection{Full kernel representation} Let $f = f_1\otimes f_2,g = g_1\otimes g_2 \in \Sigma$ be such  that for both indices $i\in\{1,2\}$ we have $\supp(f_i)\cap\supp(g_i)= \emptyset.$ Then we assume the representation
\allowdisplaybreaks \begin{align*}
	\langle Tf,g\rangle = \int_{\R^{d_1}\times\R^{d_2}}\int_{\R^{d_1}\times\R^{d_2}} K(x,y) (f_1\otimes f_2)(y)(g_1\otimes g_2)(x)\ud y\ud x,
\end{align*}
where  $K\in\CZ((d_1,d_2),\delta).$
Note that this implies the analogous kernel representations for $T^{1*},T^{2*},T^*.$

\subsubsection{Partial kernel representations} Now, let $f = f_1\otimes f_2,g = g_1\otimes g_2 \in \Sigma$ be such  that for one index $j\in\{1,2\}$ we have $\supp(f_j)\cap\supp(g_j)= \emptyset.$ Then, we assume the representation 
\allowdisplaybreaks \begin{align*}
	\langle T(f_1\otimes f_2),g_1\otimes g_2\rangle = \int_{\R^{d_j}}\int_{\R^{d_j}} K_{f_i,g_i}(x_j,y_j)f_j(y_j)g_j(x_j)\ud y_j\ud x_j,
\end{align*}
where $K_{f_i,g_i} \in \CZ(\delta,d_j)$ is such that $\Norm{K_{f_i,g_i}}{\CZ(\delta,d_j)} \leq C(f_i,g_i)$ for some positive constant that depends on the functions $f_i,g_i.$ We also assume these constants to have the following bounds 
\allowdisplaybreaks \begin{align*}
	C(1_P,1_P)+C(1_P,a_P) + C(a_P,1_P) \leq C\abs{P}
\end{align*}
for all functions $a_P\in\Sigma_i$ such that  $a_P = 1_Pa_P,$ $\abs{a_P} \leq 1,$ and $\int a_P = 0,$ where $P$ is a cube on $\R^{d_i}.$

\begin{defn} A linear operator $T$ with the full and partial kernel representations  as described in this section, is called a bi-parameter singular integral operator.
\end{defn}

\begin{defn}\label{defn:CZO:mart} A bi-parameter singular integral operator $T$ such that $\Norm{T}{L^p(\R^d)\to L^p(\R^d)} + \Norm{T^{1*}}{L^p(\R^d)\to L^p(\R^d)}< \infty$ for some $p\in(1,\infty)$ (equivalently, for all $p$) is called a bi-parameter Calderón-Zygmund operator.
\end{defn}

%\begin{rem} A bi-parameter SIO is a bi-parameter CZO if and only if the bi-paramater $T1$ conditions are satisfied, see Journé \cite{Jo} and Martikainen \cite{Ma1}.
%\end{rem}

\subsection{Basic notation}
When we consider a bi-parameter product space $\R^d = \R^{d_1} \times \R^{d_2}$ we often denote the mixed-norm space $L^{p_1}(\R^{d_1}; L^{p_2}(\R^{d_2}))$
by $L^{p_1}_{x_1} L^{p_2}_{x_2}.$
We identify $f \colon \R^d \to \C$ satisfying
$$
\Big(\int_{\R^{d_1}} \Big( \int_{\R^{d_2}} |f(x_1, x_2)|^{p_2} \ud x_2 \Big)^{p_1/p_2} \ud x_1\Big)^{1/p_1} < \infty
$$
with the function $\phi_f \in L^{p_1}(\R^{d_1}; L^{p_2}(\R^{d_2}))$, $\phi_f(x_1) = f(x_1, \cdot)$. 

We write all identities almost everywhere. For example, if a function can be made to satisfy a property (e.g. to be a constant, or continuous, etc...) by redefining it in a set of measure zero, we say that the function satisfies that property.

We denote cubes in $\R^{d_1}$ by $I$, and cubes in $\R^{d_2}$ by $J$ --  that is, the dimension of the cube can be read from which symbol we are using. Various rectangles then take the form $I\times J.$ The side-length and the diameter of a cube $I$ are denoted respectfully by $\ell(I)$ and $\diam(I)$. Centre-points of cubes and rectangles are denote as $c_Q, c_R.$

Often integral pairings need to be taken with respect to one of the variables only.
For example, if $f \colon \R^{d_1}\times\R^{d_2} \to \C$ and $h_I \colon \R^{d_1} \to \C$, then $\langle f, h_I \rangle \colon \R^{d_2} \to \C$ is defined by
$$
\langle f, h_I \rangle(x_2) = \int_{\R^{d_1}} f(y_1, x_2)h(y_1)\ud y_1.
$$
On several occasions we use operators that only act on one of the variables, e.g. the maximal function $\mathsf{M}:L^p_{x_2}\to L^p_{x_2}$ and we denote it acting on a function of two parameters as $\mathsf{M}f(x_1,x_2) = \mathsf{M}(f(x_1,\cdot))(x_2)$.
If unclear on what parameter slots these auxiliary operators are acting, we denote $\mathsf{M}^i,\mathsf{M}^{\alpha_i},$ etc.
%
%We denote averages by
%$$
%\langle f \rangle_A = \fint_A f:= \frac{1}{|A|} \int_A f,
%$$
%where $|A|$ denotes the Lebesgue measure of the set $A$. The indicator function of a set $A$ is denoted by $1_A$. 

Throughout the exponents $p_1,p_2,q_1,q_2$ will  always be in the range $(1,\infty)$ but this will not always be mentioned. We will sometimes write $p=(p_1,p_2)$ and $q =(q_1,q_2)$ to shorten notation and this will be clear from the context.

We denote $A \lesssim B$, if $A \leq C B$ for some constant $C>0$ depending only on the dimension of the underlying space, on the integrability exponents and on other unimportant absolute constants appearing in the assumptions.
Then $A \sim B$, if $A \lesssim B$ and $B \lesssim  A.$ Subscripts on constants ($C_{a,b,c,...}$) and quantifiers ($\lesssim_{a,b,c,...}$) signify their dependence on those subscripts.

\section{Approximate weak factorization in the bi-parameter setting}\label{sect:awf}
We will next go through the awf argument for proving commutator lower bounds in the bi-parameter setting. We refer the reader to consult \cite{HyLpLq} for a lengthier discussion in the standard one-parameter setting. Still, let us recall some important points.
%, however, we shortly present the main idea also here in the diagonal case.
%Given a function $f$ chosen so that the first identity of the following line is satisfied, and given that we could factor it according to the latter identities,
%\begin{align*}
%	\ave{\abs{b-\ave{b}_R}}_R = \langle b,f\rangle = \langle b, hTg-gT^*h \rangle = \langle [b,T]g,h\rangle,
%\end{align*}
%where the functions $g,h$ are somehow related to $f$ and $T$, then, we could hope to have estimates for $\Norm{b}{\BMO}$ in terms of the commutator $[b,T].$
%However, this is not the kind of expansion we get, and instead, we have an approximate weak-factorization where an additional error term $\tilde{f}$ remains over. However, this error term can be shown to be small and hence it can be handled via an absorbtion argument. 

When a commutator lower bound is proved, the full norm $\|[b,T]]\|_{L^{p_1}_{x_1} L^{p_2}_{x_2} \to L^{q_1}_{x_1} L^{q_2}_{x_2}}$ is not actually needed but so-called off-support versions of the norm we denote as $
\mathcal{O}_{p,q}(b;K)$ and  $\mathcal{O}_{p,q}^{\Sigma}(b;K)$
are used and these can be defined even if we only have $b \in L^1_{\loc}.$ Indeed, in defining these off-support norms what we use is the assumption
\allowdisplaybreaks \begin{align*}
Tf(x)=\int_{\R^d}K(x,y)f(y)\ud y, \quad x\not\in \supp(f),
\end{align*}
and this only involves the kernel. It is actually true in all cases that we are estimating the size of the off-support norms via testing conditions on $b$ more than just simply the size of the full norm. Consequently, where we achieve a full characterization we also obtain as immediate corollaries the information
\allowdisplaybreaks \begin{align}\label{aaa}
\mathcal{O}_{(p_1,p_2),(q_1,q_2)}(b;K) \sim  	\|[b,T]\|_{L^{p_1}_{x_1} L^{p_2}_{x_2} \to L^{q_1}_{x_1} L^{q_2}_{x_2}}.
\end{align}
Here we understand that the left-hand side of \eqref{aaa} is defined for $b\in L^1_{\loc}$ and the kernel $K,$ while when we write the right-hand side, we assume implicitly that the commutator $[b,T]$ is well-defined and bounded.

At the heart of the business lies the notion of non-degeneracy.
\begin{defn}\label{defn:kernel:1} A bi-parameter kernel $K$ is called non-degenerate, if for each $x = (x_1,x_2) \in\R^d$ and two radii $r_1,r_2>0,$ there exists $y = (y_1,y_2)$ such that
	\[
	\abs{K(x,y)} \gtrsim r_1^{-d_1}r_2^{-d_2},\qquad \abs{x_1-y_1}>r_1,\qquad \abs{x_2-y_2}>r_2.
	\]
\end{defn}

To obtain commutator lower bounds, we will also assume that the kernel $K$ satisfies the size estimate \eqref{ker:size} 
and the mixed size-regularity conditions
\allowdisplaybreaks \begin{align}\label{kernel:regularity}
\abs{K((x_1,x_2),y)-K((x_1',x_2),y)} \leq C \frac{1}{\abs{x_1-y_1}^{d_1}}\omega_1\bigg(\frac{\abs{x_1-x_1'}}{\abs{x_1-y_1}}\bigg)\frac{1}{\abs{x_2-y_2}^{d_2}},
\end{align}
whenever $\abs{x_1-x_1'}\leq 1/2\abs{x_1-y_1},$ of which we also have the three other variants.

Notice that given the points $x,y$ as in Definition \ref{defn:kernel:1}, it follows from the size estimate that
\begin{align}\label{a}
	r_1^{-d_1} r_2^{-d_2} \lesssim |K(x,y)| \lesssim |x_1-y_1|^{-d_1} |x_2-y_2|^{-d_2} \lesssim |x_1-y_1|^{-d_1} r_2^{-d_2},
\end{align}
hence $|x_1-y_1| \lesssim r_1,$ and similarly we see that $|x_2-y_2| \lesssim r_2,$ and consequently that
\begin{align}\label{aa}
	\abs{x_i-y_i}\sim r_i,\qquad i=1,2.
\end{align}

Of the functions $\omega_i$ appearing the mixed- and full regularity estimates we ask that they are increasing, subadditive and satisfy $\omega_i(\alpha) \to 0$ as $\alpha \to 0.$ We will use a single function $\omega$ to deal with all the parameter slots, as we have $
\omega_i \leq \max_{i\in \{1,2,3,4\}}\omega_i  =:\omega, $ 
and $\omega$ is a function that satisfies the same assumptions as each single $\omega_i.$ 

Obviously the class of standard bi-parameter CZ-kernels is encompassed here, however, it is a larger class in another sense also: we do not require any kind of full regularity conditions, see section \ref{section:biparSIOs}. 

%This is due to the fact that Proposition \ref{bootstrap} to which the rest of our argument bootstraps, does not require the regular. We only use the mixed size-regularity bounds.

\begin{prop}\label{bootstrap} Let $K$ be a non-degenerate bi-parameter kernel as in Definition \ref{defn:kernel:1} that satisfies the size estimate \eqref{ker:size} and the mixed size-regularity estimates \eqref{kernel:regularity}. 
	
	Fix a constant $A\geq 3$ and let $R= I\times J$ be a rectangle. Then, there exists a rectangle $\widetilde{R} = \wt{I}\times\wt{J}$ of the same dimensions as $R,$ i.e. $\ell(\wt{I}) = \ell(I)$ and $\ell(\wt{J}) = \ell(J),$ localized as
	\begin{align}\label{asdfgh}
		 \dist(I,\widetilde{I}) \sim A\diam(I) \qquad 	\dist(J,\widetilde{J}) \sim A\diam(J)
	\end{align}
	and which satisfies the following: for all $x\in R$ and $y\in\widetilde{R}$ we have that
\begin{align}\label{kkk}
		\abs{K\big(x,y\big) - K\big( c_R,c_{\wt{R}}  \big)} \lesssim  A^{-(d_1+d_2)}\abs{R}^{-1}\omega\big(1/A \big),
\end{align}
and if we choose $A$ large enough, we also have, 
\allowdisplaybreaks \begin{align*}
\Babs{\int_RK(x,y)\ud x}\sim  \Babs{\int_{\wt{R}}K(x,y)\ud y}   \sim \int_R\abs{K(x,y)} \ud x \sim  \int_{\wt{R}}\abs{K(x,y)}\ud y \sim A^{-(d_1+d_2)}.
\end{align*}
\end{prop}
\begin{proof}
	Let $c_{R}=(c_{I},c_{J})\in\R^{d_1+d_2}$ be the centre of a rectangle $R.$ By the non-degeneracy of $K,$ we find a point $c_{\wt{R}}=(c_{\wt{I}},c_{\wt{J}})$ such that
	\begin{align}\label{asdf}
			\abs{c_I-c_{\wt{I}}} \geq A\ell(I),\quad 	\abs{c_J-c_{\wt{J}}} \geq A\ell(J)
	\end{align}
	 that is the centre of a rectangle $\wt{R} = \wt{I}\times\wt{J}$ and satisfies
	\allowdisplaybreaks \begin{align*}
	\abs{K(c_R,c_{\wt{R}})} \gtrsim A^{-(d_1+d_2)} \ell(I)^{-d_1}\ell(J)^{-d_2} =A^{-(d_1+d_2)} \abs{R}^{-1}.
	\end{align*}
	The claims on the line \eqref{asdfgh} follow immediately from the remarks following the definition of non-degeneracy, see the lines \eqref{a} and \eqref{aa}. Moreover, by the size estimate and \eqref{asdf} we have that $\abs{K(c_R,c_{\wt{R}})} \lesssim A^{-(d_1+d_2)} \abs{R}^{-1}$ and consequently that
	\begin{align}\label{csize}
	\abs{K(c_R,c_{\wt{R}})} \sim A^{-(d_1+d_2)} \abs{R}^{-1}.
	\end{align}
	Now let $x\in R$ and $y\in\wt{R}$ be arbitrary. To see why \eqref{kkk} holds, we use the mixed size-regularity conditions \eqref{kernel:regularity}. We have
	\allowdisplaybreaks \begin{align*}
	\begin{split}
	\abs{K\big(x,y\big) - K\big( c_R,c_{\wt{R}}  \big)} &\leq \abs{K\big( [ x_1,x_2 ]  , [y_1,y_2 ]  \big) - K\big( [ c_{I},x_2 ]  , [y_1 , y_2 ]  \big)} \\
	&+ \abs{K\big( [c_{I} ,x_2 ]  , [y_1 ,y_2 ]  \big) - K\big( [ c_{I} ,c_{J} ]  , [y_1 ,y_2 ]  \big)}\\
	&+ \abs{K\big( [c_{I} , c_{J} ]  , [ y_1 , y_2 ]  \big) - K\big( [c_{I} ,c_{J} ]  , [c_{\wt{{I}}} , y_2 ]  \big)} \\
	&+ \abs{K\big( [c_{I} ,c_{J} ]  , [c_{\wt{{I}}} , y_2 ]  \big) - K\big( [ c_{I} , c_{J} ]  , [ c_{\wt{{I}}},c_{\wt{{J}}} ]  \big)} \\
	& \lesssim A^{-(d_1+d_2)}\abs{R}^{-1}\omega\big(1/A \big),
	\end{split}
	\end{align*}
	where for example the estimate for the first of the four intermediate terms derives as 
	\allowdisplaybreaks \begin{align*}
	\abs{K( [ x_1,x_2 ]  , [y_1,y_2 ]  ) - K( [ c_{I},x_2 ]  , [y_1 , y_2 ]  )} &\lesssim  \frac{1}{\abs{c_I-y_1}^{d_1}}\omega\bigg(\frac{\abs{x_1-c_{I}}}{\abs{c_I-y_1}}\bigg)\frac{1}{\abs{x_2-y_2}^{d_2}}\\
	&\lesssim A^{-d_1}\ell({I})^{-d_1}\omega\big(C/A\big)A^{-d_2}\ell({J})^{-d_2} \\ 
	&\lesssim A^{-(d_1+d_2)}\abs{R}^{-1}\omega\big(1/A\big),
	\end{align*}
	where used the fact that $A\geq3 $ to apply the mixed size-regularity estimates and the sub-additivity of $\omega.$
	
	Now, the last four claims involving the integrals follow by choosing $A$ sufficiently large, by subtracting and adding $K(c_R,c_{\wt{R}}),$ and using the estimates \eqref{kkk} and \eqref{csize}.
\end{proof}

\begin{prop}\label{wf:2par}   Let $K$ be a non-degenarate bi-parameter kernel as in Proposition \ref{bootstrap}. Let $R=I\times J$ be a fixed rectangle and let $f$ be a locally integrable function such that $\supp(f)\subset R,$ $\int f = 0.$
	
	Then, for a choice of the constant $A$ large enough, the function $f$ can be written as
	\allowdisplaybreaks 
	\begin{align}\label{fact}
		f =  [h_1Tg_1 -g_1 T^*h_1] + [ h_2T^*g_2- g_2 Th_2] +  \tilde{f},
	\end{align}
	where the appearing auxiliary functions satisfy 
	\begin{align}\label{local1}
	 g_1 = 1_{\wt{R}},\qquad g_2 = 1_R,\qquad 	\supp(h_1)\subset R,\qquad \supp(h_2)\subset \wt{R},\qquad \supp(\tilde{f})\subset R
	\end{align}
	and 
	\begin{align}\label{local2}
			\abs{h_1(x)} \lesssim A^{d}\abs{f(x)},\qquad \abs{h_2(x)} \lesssim A^d\ave{\abs{f}}_R1_{\widetilde{R}}(x),\qquad 
		\abs{\tilde{f}(x)} \lesssim \omega(\frac{1}{A})\ave{\abs{f}}_R1_{R}(x),
	\end{align}
	and we have $\int\tilde{f}=0.$  
\end{prop}
\begin{proof}
 Let $\wt{R}=\wt{I}\times\wt{J}$ be the rectangle as obtained by Proposition \ref{bootstrap} and let $g_1 := 1_{\wt R}$. We decompose the function $f$ as
	\allowdisplaybreaks \begin{align*}
	f= h_1Tg_1 -g_1 T^*h_1 + \tilde{f},\qquad h_1= \frac{f}{Tg_1},\quad \tilde{w} = g_1 T^*h_1.
	\end{align*}
	The only problem with the above factorization is that $h_1$ might a priori involve a division by zero, the following estimates show that this is not the case.  Let $x\in R,$ then 
	\allowdisplaybreaks \begin{align*}
	Tg_1(x)  &=\int_{\wt{R}} K(x,y) \ud y = \int_{\wt{R}}(K(x,y)-K(c_R,c_{\wt{R}}))\ud y + \int_{\wt{R}} K(c_R,c_{\wt{R}})\ud y \\
	&= I + II.
	\end{align*}
  	It follows by Proposition \ref{bootstrap} that
	\allowdisplaybreaks \begin{align*}
	\abs{ I } \lesssim A^{-(d_1+d_2)}\omega(1/A),\qquad
	\abs{II} \sim A^{-(d_1+d_2)}
	\end{align*}
	and hence for $A$ sufficiently large that $\abs{Tg_1(x)} \sim A^{-(d_1+d_2)},$ making $h_1$ well-defined. Also, by the above we have
	\allowdisplaybreaks \begin{align*}
	\abs{h_1(x)} \lesssim A^{d_1+d_2}\abs{f(x)},
	\end{align*}
	which establishes the left-most estimate on the line \eqref{local2}.
	Then, to estimate the first iteration error term $\tilde{w},$ let $y\in R$ and write
	\allowdisplaybreaks \begin{align*}
	\begin{split}
	\frac{f}{Tg_1}(y) &= \Big(\frac{f}{Tg_1} - \frac{f}{\int_{\wt{R}} K(c_R,c_{\wt{R}})\ud z} \Big)(y) + \frac{f(y)}{\int_{\wt{R}} K(c_R,c_{\wt{R}})\ud z}= III + IV.
	\end{split}
	\end{align*}
	By Proposition \ref{bootstrap} it follows that
	\allowdisplaybreaks \begin{align*}
	\begin{split}
	\abs{III} &= \Babs{ f(y)\int_{\wt{R}} ( K(y,z)-K(c_R,c_{\wt{R}}))\ud z }\times\Babs{\int_{\wt{R}}K(y,z)\ud z\int_{\wt{R}} K(c_R,c_{\wt{R}})\ud z}^{-1}  \\ 
	&\lesssim \frac{\abs{f(y)}A^{-d}\omega(\frac{1}{A})}{A^{-d}A^{-d}} 
	= A^{d}\omega\Big(\frac{1}{A}\Big)\abs{f(y)},
	\end{split}
	\end{align*}
	 and hence, we have with $x\in\wt{R}$ that 
	\allowdisplaybreaks \begin{align*}
	\begin{split}
	&\Babs{ T^*\Big(\frac{f}{Tg_1} - \frac{f}{\int_{\wt{R}} K(c_R,c_{\wt{R}})\ud y} \Big)}(x) \lesssim A^{d}\omega\Big(\frac{1}{A}\Big) \cdot\int_{R}\abs{K(y,x)}\abs{f(y)}\ud y \\
	&\lesssim A^{d}\omega\Big(\frac{1}{A}\Big) \cdot A^{-d}\ave{\abs{f}}_R1_{\widetilde{R}}(x) = \omega\Big(\frac{1}{A}\Big) \cdot \ave{\abs{f}}_R1_{\widetilde{R}}(x),
	\end{split}
	\end{align*}
	where we simply used the size estimate.
	
	By the zero mean of $f$ on $R$ we have
	\allowdisplaybreaks \begin{align*}
	\begin{split}
	&\babs{T^*(IV)(y)} = \Babs{T^*\Big(\frac{f}{\int_{\wt{R}} K(c_R,c_{\wt{R}})\ud y}\Big)(x)} = \frac{\Babs{\int_R (K(y,x)-K(c_{\wt{R}},c_R)f(y)\ud y}}{\Babs{\int_{\wt{R}} K(c_R,c_{\wt{R}})\ud y}} \\
	& \lesssim  \omega\Big(\frac{1}{A}\Big)A^{-d}\ave{\abs{f}}_R 1_{\widetilde{R}}(x) \cdot A^{d} = \omega\Big(\frac{1}{A}\Big) \cdot \ave{\abs{f}}_R1_{\widetilde{R}}(x),
	\end{split}
	\end{align*}
	where we used the mixed size-regularity estimates and Proposition \ref{bootstrap}.
	Hence, combining the above parts, we obtain
	\allowdisplaybreaks \begin{align*}
	\abs{\tilde{w}(x)} \lesssim  \omega\Big(\frac{1}{A}\Big)\ave{\abs{f}}_R1_{\widetilde{R}}(x).
	\end{align*}
	It is also immediate from the definitions that
	\begin{align}\label{zeromean}
		\int_{\wt{R}}\tilde{w} = \int g_1 T^*\big(\frac{f}{Tg_1}\big) = \int Tg_1 \frac{f}{Tg_1} =  \int f = 0.
	\end{align}
	
	Now, let $g_2 = 1_R.$ By repeating the above argument, but now starting with the function $\tilde{f}$ supported on the rectangle $\wt{R}$ we write
	\allowdisplaybreaks \begin{align*}
	\tilde{w} = h_2T^*g_2- g_2 Th_2 +\tilde{\tilde{f}},\qquad h_2 = \frac{\tilde{w}}{Tg_2},\quad  \tilde{f} =  g_2Th_2.
	\end{align*}
	 With the same arguments and proofs as above, the function $h_2$ is well-defined and for $x\in \widetilde{R}$ we have that
	$$\abs{h_2(x)}\lesssim A^{d}\abs{\tilde{f}(x)} \lesssim  A^{d}\omega\Big(\frac{1}{A}\Big)\ave{\abs{f}}_R1_{\widetilde{R}}(x) \lesssim A^{d} \ave{\abs{f}}_R1_{\widetilde{R}}(x)$$
	and for $x\in R,$ with $A$ large enough, that
	\[
	\abs{\tilde{f}(x)} \lesssim \omega\Big(\frac{1}{A}\Big)\cdot \ave{\abs{\tilde{f}}}_{\widetilde{R}}1_R(x) \lesssim \omega\Big(\frac{1}{A}\Big)^2\ave{\abs{f}}_R1_R(x)\lesssim \omega\Big(\frac{1}{A}\Big)\ave{\abs{f}}_R1_R(x).
	\]
	Moreover, as in \eqref{zeromean}, the secont iteration error term $\wt{f}$ inherits the zero mean from $\wt{w}.$
\end{proof} 

	Let us notate the oscillation of a function $b \in L^1_{\loc}$ over a rectangle $R=I\times J$ with
	\allowdisplaybreaks \begin{align*}
	\osc(b;R) =  \fint_R\abs{b-\ave{b}_R}.
	\end{align*}

\begin{prop}\label{prop:oscform}	Let $K$ be a symmetrically non-degenerate bi-parameter kernel and $b\in L^1_{\loc}.$
	Then, for all rectangles $R=I\times J$ we have
\begin{align*}
	\abs{R}\osc(b;R) \lesssim \abs{\langle [b,T]g_1,h_1\rangle} + \abs{\langle [b,T]h_2,g_2\rangle},
\end{align*}
where the appearing functions are as in \ref{wf:2par},
\[
g_1 = 1_{\wt{R}},\quad g_2 = 1_{R},\qquad 	\abs{h_1(x)} \lesssim_A 1_R(x),\quad \abs{h_2(x)} \lesssim_A 1_{\widetilde{R}}(x).
\] 
\end{prop}
\begin{proof} As $b\in L^1_{\loc}$, we find a function $f$ of zero mean supported on $R$ such that $\Norm{f}{\infty}\leq 1$ and
	\[
	\abs{R}\osc(b;R) \sim \int bf.
	\]
By Proposition \ref{wf:2par} we write and estimate the right-hand side as 
	\begin{align*}
		\int bf & = \int b[h_1Tg_1 -g_1 T^*h_1] + \int b [ h_2T^*g_2- g_2 Th_2] + \int b \tilde{f} \\
		&\leq \abs{\langle [b,T]g_1,h_1\rangle} + \abs{\langle [b,T]h_2,g_2\rangle} + 	\babs{\int b \tilde{f}}
	\end{align*}
and the error term further to
\begin{align*}
		\babs{\int b \tilde{f}} \leq \Norm{\tilde{f}}{\infty}\int_R\babs{b-\ave{b}_R}\lesssim \omega(1/A)\abs{R}\osc(b;R).
\end{align*}
By having the above estimates together we obtain
\begin{align}\label{absorb}
		\abs{R}\osc(b;R) \lesssim \abs{\langle [b,T]g_1,h_1\rangle} + \abs{\langle [b,T]h_2,g_2\rangle} + \omega(1/A)\abs{R}\osc(b;R).
\end{align}
As $b\in L^1_{\loc},$ by choosing $A$ large enough, we absorb the common term in \eqref{absorb} to the left-hand side.
\end{proof}

The first off-support norm we use is
\begin{defn}\label{norm:1} Let $b\in L^1_{\loc}$ and define
	\allowdisplaybreaks \begin{align*}
	\mathcal{O}_{(t_1, t_2), (s_1, s_2)}^{A}(b;K) =  \sup_{\substack{R= I\times J \\ \wt{R} = \wt{I}\times\wt{J}}} \Big|  \frac{\int_{\R^d\times \R^d}(b(x)-b(y)) K(x,y)f(y)g(x) \ud y \ud x}{\abs{I}^{1/t_1+1/s_1'}\abs{J}^{1/t_2+1/s_2'}} \Big|,
	\end{align*}
	where the supremum is taken over rectangles $R = I \times J$ and $\wt R = \wt I \times \wt J$ with
	$$
	\dist (I, \wt I)\sim A \ell(I) \qquad \textup{and} \qquad \dist (J, \wt J)\sim A \ell(J)
	$$
	and over functions $f \in L^{\infty}(R)$ and $g \in L^{\infty}(\wt R)$ with
	$$
	\|f\|_{L^{\infty}} \le 1 \qquad \textup{and} \qquad  \|g\|_{L^{\infty}} \le 1.
	$$
\end{defn}

\begin{rem}	
	When $p=(p_1,p_2),q=(q_1,q_2)$ we may write $\mathcal{O}_{p,q}^{A}(b;K)  = \mathcal{O}_{(p_1,p_2),(q_1,q_2)}^{A}(b;K).$
	
	From this point onwards we will fix the constant $A$ large enough so that we may always use the conclusions of all the above stated propositions where the constant $A$ appears and we will drop the superscript $A$ and simply write $\mathcal{O}_{p,q}(b;K).$
\end{rem}

Relating the oscillation to the off-support norm, we have the following
\begin{prop}\label{prop:oscupperbound}
	Let $K$ be a non-degenerate bi-parameter kernel, $b\in L^1_{\loc}$ and $s_i,t_i\in(1,\infty).$
	Then, for all rectangles $R=I\times J$ we have 
	\allowdisplaybreaks \begin{align*}
	\osc(b;R)  \lesssim \mathcal{O}_{(t_1,t_2),(s_1,s_2)}(b;K)\abs{I}^{1/t_1-1/s_1}\abs{J}^{1/t_2-1/s_2}.
	\end{align*}
\end{prop}
\begin{proof} By Proposition \ref{prop:oscform} we write
	\[
		\abs{R}\osc(b;R)  \lesssim \abs{\langle [b,T]g_1,h_1\rangle} + \abs{\langle [b,T]h_2,g_2\rangle},
	\]
	for functions $h_i,g_i$ as in Proposition \ref{prop:oscform}. By the definition of the off-support norm we estimate
	\allowdisplaybreaks \begin{align*}
	\abs{\langle [b,T]g_1,h_1\rangle} &\leq  \mathcal{O}_{(t_1,t_2)(s_1,s_2)}\abs{I}^{1/t_1+1/s_1'}\abs{J}^{1/t_2+1/s_2'} \\
	 &= \abs{R}\mathcal{O}_{(t_1,t_2),(s_1,s_2)}(b;K)\abs{I}^{1/t_1-1/s_1}\abs{J}^{1/t_2-1/s_2}
	\end{align*}
	and similarly for the other term. 
	Dividing with $\abs{R},$ the claimed estimate follows.
\end{proof}

\section{Upper and lower bounds}
In this section we prove all the stated lower bounds and those upper bounds that admit a short proof, with the remaining upper bounds postponed to sections \ref{section:alphahöldercases1} and \ref{sect:last}.

\subsection{The case $p_i=q_i>1,$ $i=1,2$ } This case is not new, other proofs are contained e.g. in \cite{HPW2018} and \cite{LMV2019biparcom} both that treat the problem in the Bloom setup. Given the awf argument prepared in the previous section, the arguments are shortly stated and we gather them here in the unweighted off-diagonal setting in Proposition \ref{prop:bicom-bmo}. Also, when in addition all the exponents are the same, we record how to derive the bi-parameter Bloom type lower bound directly from Proposition \ref{prop:oscupperbound}. Proposition \ref{prop:biparBloom} is not new either, the special case with the Riesz transforms is contained in \cite{HPW2018} and the result with the same non-degeneracy assumptions as we use is in \cite{LMV2019biparcom}. In \cite{LMV2019biparcom} the far simpler median method is used which limits their considerations to real-valued functions $b,$ on the other hand the median method works also for iterated commutator. On the other side, by the awf argument, we can consider complex-valued functions $b,$ however we have no hope of characterizing the iterated cases.

\begin{prop}\label{prop:bicom-bmo} Let $1<p_i=q_i<\infty,$ $i=1,2$ and assume that $b\in L^1_{loc}.$ Then, 
\allowdisplaybreaks \begin{align*}
	\Norm{b}{\bmo(\R^{d_1}\times\R^{d_2})} \lesssim \mathcal{O}_{(p_1,p_2),(p_1,p_2)}(b;K) \leq \|[b,T]]\|_{L^{p_1}_{x_1} L^{p_2}_{x_2} \to L^{p_1}_{x_1} L^{p_2}_{x_2}} \lesssim \Norm{b}{\bmo(\R^{d_1}\times\R^{d_2})} 
\end{align*}
\end{prop}
\begin{proof}  The first estimate is immediate from \ref{prop:oscupperbound}, while the second follows by a simple application of Hölders' inequality. 
	Hence, the only claim left to show is the upper bound
\allowdisplaybreaks \begin{align*}\label{jjj}
	 \Norm{[b,T]}{L^{p_1}_{x_1}L^{p_2}_{x_2}\to L^{p_1}_{x_1}L^{p_2}_{x_2}} \lesssim 	\Norm{b}{\bmo(\R^{d_1}\times\R^{d_2})}.
\end{align*}
This, is proved with exactly the same argument as the commutator upper bounds are proved in \cite{LMV:Bloom}, the fact that we have mixed norms appear, contrary to the non-mixed cases, plays no significant role in the proof at all.
\end{proof}

Let us then turn to the Bloom type lower bound.
Recall that a positive function $\mu$ is in the bi-parameter $A_p$ if 
\[
[\mu]_{A_p(\R^{d_1}\times\R^{d_2})} := \sup_{R}\ave{\mu}_R\ave{\mu^{-\frac{p'}{p}}}_R^{\frac{p}{p'}}<\infty,
\]
and for a positive locally integral function $\nu$ we write $b\in\bmo_{\nu},$ if 
\[
\Norm{b}{\bmo_{\nu}} := \sup_R\frac{1}{\nu(R)}\int_R\abs{b-\ave{b}_R}<\infty,\quad \nu(R) = \int_R\nu.
\]
Notice that if we have two weights $\lambda,\mu\in A_p,$ then by a simple application of Hölder's inequality we have that $\nu = (\mu/\lambda)^{1/p}\in A_2.$

Also in the Bloom case, we use an off-support norm. The only difference compared to Definition \ref{norm:1} is that now the normalization is modified and we consider the quantity
 \[
  \mathcal{O}_{(p,p)}(b;K;\mu;\lambda) =  \sup \Big|  \frac{\int_{\R^d\times \R^d}(b(x)-b(y)) K(x,y)f(y)g(x) \ud y \ud x}{ \mu(R)^{1/p} [\lambda^{-\frac{p'}{p}}(\wt{R})]^{1/p'}  } \Big|
 \]
 where the supremum is taken over all such functions $f$ and $g$ as in the Definition \ref{norm:1}.

\begin{prop}\label{prop:biparBloom} Let $b\in L^1_{\loc}$ and let $\mu,\lambda$ be bi-parameter $A_p$ weights. Then, we have that
	\[
	\Norm{b}{\bmo(\nu)} \lesssim_{[\mu]_{A_p},[\lambda]_{A_p}} \mathcal{O}^A_{(p,p)}(b;K;\mu;\lambda) \leq  \Norm{[b,T]}{L^p(\mu)\to L^p(\lambda)}\lesssim_{[\mu]_{A_p},[\lambda]_{A_p}} 	\Norm{b}{\bmo(\nu)},
	\]
	where $\nu := (\frac{\mu}{\lambda})^{1/p}.$ 
\end{prop}
\begin{proof} By Proposition \ref{prop:oscform} we estimate
	\begin{align*}
	 \int_R\abs{b-\ave{b}_R}=  \abs{R}\osc(b;R) &\lesssim  \sum_{i=1,2} \abs{\langle [b,T]g_i,h_i\rangle} \\
	 &\leq \mathcal{O}_{(p,p)}(b;K;\mu;\lambda)\big( \mu(R)^{1/p} [\lambda^{-\frac{p'}{p}}(\wt{R})]^{1/p'}  \\
	 &\qquad\qquad+ \mu(\wt{R})^{1/p} [\lambda^{-\frac{p'}{p}}(R)]^{1/p'} \big).
	\end{align*}
Since $A_p$ weights are doubling and
\[
\dist(I,\wt{I})\sim \diam(I),\qquad \dist(J,\wt{J})\sim\diam(J),
\] 
it follows that 
\begin{align}\label{qwerty}
	\mu(\wt{R}) \sim_{[\mu]_{A_p}} \mu(R),\qquad  \lambda(\wt{R}) \sim_{[\lambda]_{A_p}} \lambda(R),\qquad \nu(\wt{R}) \sim_{[\mu]_{A_p},[\lambda]_{A_p}} \nu(R).
\end{align}

Hence, we estimate the left-term of the previous estimate with the index $i=1$ as
\begin{align*}
\mu(\wt{R})^{1/p} [\lambda^{-\frac{p'}{p}}(R)]^{1/p'}&= \mu(\wt{R})^{1/p} \left( \int_{R}\lambda^{-\frac{p'}{p}}\right)^{1/p'} \sim_{A,[\lambda]_{A_p}}  \mu(R)^{1/p} \left( \int_{R}\lambda^{-\frac{p'}{p}}\right)^{1/p'} \\
	&\overset{*}{\leq}  [\mu]_{A_p}^{1/p}\ave{\nu}_R \lambda(R)^{1/p} \left( \int_R\lambda^{-\frac{p'}{p}}\right)^{1/p'} \leq  [\mu]_{A_p}^{1/p} [\lambda]_{A_p}^{1/p}\nu(R) \\
	&\lesssim_{ [\mu]_{A_p}, [\lambda]_{A_p}}\nu(R)
\end{align*}
where in the estimate marked with $*$ we used that
\[
\mu(R)^{1/p}\lambda(R)^{-1/p} \leq [\mu]_{A_p}^{1/p}\ave{\nu}_R,
\]
which follows by a few applications of Hölder's inequality and a rearranging of the estimate
\[
1 \leq \ave{\nu}_R\ave{\nu^{-1}}_R\leq \ave{\nu}_R\ave{\lambda}_R^{1/p}\ave{\mu^{-\frac{p'}{p}}}_R^{1/p'} \leq \ave{\nu}_R\ave{\lambda}_R^{1/p}\ave{\mu}_R^{-1/p}[\mu]_{A_p}^{1/p}.
\]
Using the other estimates from the line \eqref{qwerty} it follows that $ \mu(R)^{1/p} [\lambda^{-\frac{p'}{p}}(\wt{R})]^{1/p'}$ satisfies the same estimate, and hence, we have shown the first estimate,
\begin{align*}
\int_R\abs{b-\ave{b}_R} \lesssim_{[\mu]_{A_p}, [\lambda]_{A_p}}
 \mathcal{O}_{(p,p)}(b;K;\mu;\lambda)\nu(R).
\end{align*}

For the middle estimate, by Hölder's inequality we immediately have that
\[
\mathcal{O}^A_{(p,p)}(b;K;\mu;\lambda) \leq \Norm{[b,T]}{L^p(\mu)\to L^p(\lambda)},
\]
and the right-most estimate $\Norm{[b,T]}{L^p(\mu)\to L^p(\lambda)}\lesssim \Norm{b}{\bmo(\R^{d_1}\times\R^{d_2})}$ is proved in exactly the stated form in \cite{LMV2019biparcom}.
\end{proof}

\subsection{The three cases $p_i<q_i,$ $i=1,2$ and $p_1<q_1,$ $p_2>q_2,$ and  $p_2<q_2,$ $p_1>q_1$} In these three cases we find that the commutator is bounded if and only if $b$ is a constant function almost everywhere. By redefining $b$ in a set of measure zero we may assume that $b$ is a constant.

\begin{prop} Let $b\in L^1_{\loc},$  $p_i<q_i,$ $i=1,2,$ and assume that $\mathcal{O}_{p,q}(b;K) <\infty.$ Then, $b$ is a constant.  Conversely, if  $b$ is a constant, then $[b,T] = 0.$
\end{prop}
\begin{proof} Only one direction is non-trivial. Fix a point $x_2\in\R^d$ and consider a sequence of cubes $\R^{d_2}\supset J_k\to \{x_2\}.$ The Lebesgue differentiation theorem shows that
	\allowdisplaybreaks \begin{align*}
		\fint_I\abs{b(x_1,x_2) - \ave{b(\cdot,x_2)}_I}\ud x_1 = \lim_{k\to \infty }\osc(b;I\times J_k)
	\end{align*}
	for almost every  $x_1\in \R^{d_1}.$
	By Proposition \ref{prop:oscupperbound} we dominate the right-hand side with
	\allowdisplaybreaks \begin{align*}
		 \mathcal{O}_{p,q}(b;K) \abs{I}^{1/p_1-1/q_1}\lim_{k\to\infty}\abs{J_k}^{1/p_2-1/q_2} = 0,
	\end{align*}
where in the last step we used that $1/p_2-1/q_2>0.$
This shows that $b(\cdot,x_2)$ is a constant on all cubes $I\subset\R^{d_1},$ hence on $\R^{d_1}.$  Similarly we see that $b(x_1,\cdot)$ is a constant almost everywhere on $\R^{d_2}$. It follows that $b$ is a constant almost everywhere. 
\end{proof}

\begin{prop} Let $p_1<q_1,$ $p_2>q_2$ and assume that $\mathcal{O}_{p,q}(b;K)<\infty.$ Then, $b$ is a constant. Conversely, if  $b$ is a constant, then $[b,T] = 0.$
\end{prop}
\begin{proof}
	 By the same argument as above we see that $b(x_1,\cdot)$ is a constant and hence for any choice of a cube $J\subset\R^{d_2}$ we have that
	\allowdisplaybreaks \begin{align*}
\fint_I\abs{b(x_1,x_2) - \ave{b(\cdot,x_2)}_I}\ud x_1 &=  \fint_I\fint_J\abs{b-\ave{b}_{I\times J}}  \\
	&\lesssim \mathcal{O}_{p,q}(b;K) \abs{I}^{1/p_1-1/q_1}\abs{J}^{1/p_2-1/q_2}.
	\end{align*}
As $1/p_2-1/q_2< 0,$ letting $\abs{J}\to\infty$ shows that
	\[
	\fint_I\abs{b(x_1,x_2) - \ave{b(\cdot,x_2)}_I}\ud x_1  = 0.
	\]
	Hence, also $b(\cdot,x_2)$ is a constant and consequently $b$ is a constant.
\end{proof}

The symmetric case with a symmetric proof is
\begin{prop} Let $p_1>q_1,$ $p_2<q_2$ and assume that $\mathcal{O}_{p,q}(b;K)<\infty.$ Then, $b$ is a constant. Conversely, if  $b$ is a constant, then $[b,T] = 0.$
\end{prop}

\subsection{The cases $p_1<q_1,$ $p_2= q_2,$ and $p_2<q_2,$ $p_1=q_1$} In these cases the function $b$ is constant in one variable slot and H\"{o}lder continuous in the other. 
\begin{prop}\label{prop:lb:a1c2} Let $b\in L^1_{\loc}$, $p_1<q_1$ and $p_2=q_2$ and $\mathcal{O}_{p,q}(b;K) <\infty.$ Then,  $b(x_1,\cdot)$ is a constant  and there holds that 
	\begin{align}\label{kek}
	\Norm{b(\cdot,x_2)}{\dot C^{0,\alpha_1}_{x_1}} \lesssim \mathcal{O}_{p,q}(b;K).
	\end{align}
Conversely, if $b$ satisfies the above properties, then
$$  \Norm{[b,T]}{L^{p_1}_{x_1}L^{p_2}_{x_2}\to L^{q_1}_{x_1}L^{p_2}_{x_2}} \lesssim \Norm{b(\cdot,x_2)}{\dot C^{0,\alpha_1}_{x_1}}.$$
	
\end{prop}
\begin{proof} We see by the same argument as above that $b(x_1,\cdot)$ is a constant for almost every $x_1\in\R^{d_1},$ and by redefining in a set of measure zero, constant everywhere. Thus, for every $x_2\in\R^{d_2}$ there holds that 
	\allowdisplaybreaks \begin{align*}
&\fint_I\abs{b(x_1,x_2)-\bave{b(\cdot,x_2)}_{I}}\ud x_1 = \fint_I\fint_{J}\abs{b-\ave{b}_{I\times J}}  \\
	&\lesssim \mathcal{O}_{p,q}(b;K) \abs{I}^{1/p_1-1/q_1}\abs{J}^{1/p_2-1/p_2} = \mathcal{O}_{p,q}(b;K) \abs{I}^{1/p_1-1/q_1}
	\end{align*}
	and this implies \eqref{kek}.
	The converse direction is proved in Proposition \ref{prop:a1c2} of Section \ref{section:alphahöldercases1}.
\end{proof}

The symmetric case with a symmetric proof is
\begin{prop}\label{prop:lb:c1a2} Let $p_2<q_2$ and $p_1=q_1$ and assume that $\mathcal{O}_{p,q}(b;K) <\infty.$ Then,  $b(\cdot,x_2)$ is a constant and  
	\[
	\Norm{b(x_1,\cdot)}{\dot C^{0,\alpha_2}_{x_2}} \lesssim \mathcal{O}_{p,q}(b;K).
	\]
Conversely, if the above conclusions hold, then 
$$  \Norm{[b,T]}{L^{p_1}_{x_1}L^{p_2}_{x_2}\to L^{p_1}_{x_1}L^{q_2}_{x_2}} \lesssim 	\esssup_{x_1\in\R^{d_1}}\Norm{b(x_1,\cdot)}{\dot C^{0,\alpha_2}_{x_2}}.$$
\end{prop}
\begin{proof} The proof for the first part of the claim is completely symmetric with proof in the previous case.
	The converse direction is proved in Section \ref{sect:last}, see Proposition \ref{prop:c1a2}.
\end{proof}

Recapping, in all the above cases where we concluded the function $b$ to be a constant, we have the corresponding upper bounds as stated in Theorem \ref{thm:main} (i.e. $b=constant$ implies that $[b,T]= 0$ which implies that $N_{p,q} = 0$) and hence have constituted a full characterization of the boundedness of $[b, T],$ in these cases.  Both upper bounds for the cases where we concluded the function $b$ to be constant in one and have the Hölder continuity criterion in the other variable are lengthier and will be presented later in section \ref{section:alphahöldercases1} and \ref{sect:last}.

\subsection{The case $p_1=q_1$ and $p_2>q_2$}\label{sect:super}  We first recall some basic background.
A dyadic grid on $\R^d$ is a collection $\calD = \calD(\R^d)$ of cubes with side-lengths in the powers of two such that:
\begin{enumerate}
	\item for each $k\in\Z$ the collection $\big\{ Q\in\calD: \ell(Q) = 2^k\big\}$ is a disjoint cover of $\R^d,$
	\item for $Q,P\in\calD$ there holds that $Q\cap P\in \big\{Q,P,\emptyset\big\}.$
\end{enumerate}

Given a cube $Q,$ we let $\calD(Q)$ denote the system of dyadic cubes inside $Q$ that is attained from iteratively bisecting the sides of $Q;$
we use sparse collections made up of elements of $\calD(Q).$ 
A collection of sets $\mathscr{S}$ is said to be $\gamma$-sparse, if each $Q\in\mathscr{S}$ has a major subset $E_Q$ such that $\abs{E_Q}>\gamma\abs{Q}$ and these sets $E_Q$ are pairwise disjoint.

The stopping time family inside a fixed cube $Q_0$ is given by the following algorithm.
For a given cube $Q\in\calD$ we denote
\[
S(f;Q) = \{ P \in \calD, P \subset Q \mbox{ is maximal with }\ave{\abs{f}}_P > 2\ave{\abs{f}}_{Q} \}
\]
and let
$$\mathscr{S} = \bigcup_k\mathscr{S}_k,\qquad \mathscr{S}_{k+1} = \bigcup_{P\in \mathscr{S}_k}S(f;P),\qquad \mathscr{S}_0 = \{Q_0\}.$$

For a given collection $\mathscr{S}\subset\calD$ of dyadic cubes and for each $Q\in\mathscr{S}$ we let $\ch_{\mathscr{S}}(Q)$ consist of the maximal cubes $P\in\mathscr{S}$ such that $P\subsetneq Q.$ For a given cube $P\in\mathscr{S}$ we denote $
E_P = P\setminus \cup_{Q\in\ch_{\mathscr{S}}P}Q$, and for each $P\in\calD$ we let  $\Pi P := \Pi_{\mathscr{S}}P$ denote the minimal cube $Q$ in $\mathscr{S}$ such that $P\subset Q$ (on the condition that it exists). With this notation then, 
\[
\ch_{\mathscr{S}}(P) = \{Q\in\mathscr{S}: Q\subsetneq P,\quad \Pi Q = P\}.
\]
A variant of the following lemma is contained in \cite{HyLpLq}.

\begin{lem}\label{lem:prstop} Fix a cube $Q$ and let $f$ be a bounded function of zero mean supported on $Q.$ Then, there exists a sparse collection $\mathscr{S} = \cup_{k=1}^N\mathscr{S}_k \subset \calD_Q$ such that 
	\[
	f = \sum_{k=0}^N\sum_{P\in\mathscr{S}_k}f_P,\qquad f_P = \sum_{\Pi_{\mathscr{S}}Q = P}\Delta_Qf,
	\]
	 where the number $N$ is finite and depends only on $\Norm{f}{L^{\infty}(Q)},$ and moreover, there holds that
	\begin{enumerate}
		\item $\int f_P = 0$, for all $P\in\mathscr{S},$
		\item  $ \sum_{k=0}^N\sum_{P\in\mathscr{S}_k}\Norm{f_P}{\infty}^s1_P \lesssim_s (\mathsf{M}f)^s$ for all $s>0.$
	\end{enumerate}
\end{lem}

In the remaining lower bounds we use the off-support norm given in the following
\begin{defn}\label{norm:2} We let
\allowdisplaybreaks \begin{align*}
\mathcal{O}_{p,q}^{\Sigma,A}(b;K) =  \sup \frac{\sum_{i=1}^N  
	\Big| \iint_{\R^{d}\times \R^d} (b(x)-b(y))
	K(x,y)f_i(y)g_i(x) \ud y \ud x \Big|}{
	\Norm{\sum_{i=1}^N\Norm{f_i}{\infty}1_{R_i}}{L^{p_1}_{x_1}L^{p_2}_{x_2}}
	\Norm{\sum_{i=1}^N\Norm{g_i}{\infty}1_{\wt{R}_i}}{L^{q_1'}_{x_1}L^{q_2'}_{x_2}}
},
\end{align*}
where the supremum is taken over rectangles $R_i = I_i \times J_i$ and $\widetilde{R}_i = \widetilde{I_i}\times\widetilde{J_i}$ with
$$
\ell(L_i) = \ell(\widetilde{L_i}) \qquad \textup{and} \qquad
\dist (L_i, \widetilde{L}_i)\sim A\ell(L_i) \qquad \textup{for} \qquad L = I,J
$$
and over functions $f_i\in L^{\infty}(R_i),g_i \in L^{\infty}(\widetilde{R}_i)$, $i=1,\ldots, N$.
\end{defn}
\begin{rem}
		 Again, we will suppress the superscript $A$ from $\mathcal{O}_{p,q}^{\Sigma,A}$ and just write $\mathcal{O}_{p,q}^{\Sigma}.$
	Using that for linear operators $U$ there holds
	$$
	\sum_{i=1}^N \langle Uf_i, g_i \rangle = \E \Big\langle U\Big( \sum_{i=1}^N \eps_i f_i \Big),  \sum_{j=1}^N \eps_j g_j \Big\rangle,
	$$
	for Rademacher random signs $\varepsilon_i$,
	it follows by Hölder's inequality that $$\mathcal{O}_{p,q}^{\Sigma}(b;K) \le \Norm{[b,T]}{L^{p_1}_{x_1}L^{p_2}_{x_2} \to L^{q_1}_{x_1}L^{q_2}_{x_2}}.$$ 
	Consequently, this is a reasonable off-support constant.
\end{rem}

\begin{prop}\label{prop:infty1r2} Let $p_1=q_1,$ $q_2<p_2$ and set $1/q_2=1/r_2+1/p_2,$ and assume that  $b \in L^1_{\loc,x_1}L^{r_2}_{\loc,x_2}.$
Then, there holds that 
	\begin{align}\label{smurffi}
	\inf_{c\in\C}\Norm{b-c}{L^{\infty}_{x_1}L^{r_2}_{x_2}} \lesssim \mathcal{O}_{p,q}^{\Sigma}(b;K) \leq \Norm{[b,T]}{L^{p_1}_{x_1}L^{p_2}_{x_2}\to L^{p_1}_{x_1}L^{q_2}_{x_2}} \lesssim \inf_{c\in\C}\Norm{b-c}{L^{\infty}_{x_1}L^{r_2}_{x_2}}.
	\end{align}
\end{prop}

\begin{proof} Let $c\in\C$ be a constant and denote $\wt{b}(x_1,x_2) = b(x_1,x_2)-c.$
Then, let $f:\R^{d_2}\to \C$ be such that
\begin{align}\label{hugo}
 1_Jf = f,\qquad	\int f = 0,\qquad\Norm{f}{L^{r_2'}(\R^{d_2})} \leq 1.
\end{align}
Then, according to Lemma \ref{lem:prstop}, we let $\mathscr{S}^2$ be the sparse collection of cubes inside $J$ with respect to the function $f$ and with $R=I\times J$ write
\[
	\int_R \wt{b}f=  \int  \wt{b}\cdot 1_I\otimes\Big(\sum_{k=0}^N\sum_{P\in \mathscr{S}_k^2}  f_P\Big) =  \sum_{k=0}^N\sum_{P\in \mathscr{S}_k^2}\int \wt{b}\cdot 1_I\otimes f_P.
\] 
The last step follows from that the left-hand side is integrable and that $\sum_{P\in \mathscr{S}_k^2}\wt{b}\cdot 1_I\otimes  f_P$ are disjointly supported for each fixed $k.$
Then, as the functions $1_I\otimes f_P$ satisfy the assumptions of Proposition \ref{wf:2par} on the cubes $I\times P$ we write
	\allowdisplaybreaks \begin{align*}
	 \int \wt{b}\cdot 1_I\otimes f_P  =  \big\langle [b,T]g_{\wt{P}},h_P \big\rangle+  \big\langle [b,T]h_{\wt{P}},g_P \big\rangle + \int \wt{b}\wt{f}_P,
	\end{align*}
	where in line with Proposition \ref{wf:2par} we notate $g_1 = g_{\wt{P}},$ $g_2 = g_P,$ $h_1 = h_P,$ $h_2 = h_{\wt{P}},$ and where we use that the commutator annihilates constants to change $\wt{b}$ to $b$ inside the commutator.
Consequently,
\begin{equation}\label{line1}
	\begin{split}
		\Babs{	\int_R \wt{b}f} &= \Babs{ \sum_{k=0}^N\sum_{P\in \mathscr{S}_k^2}  \big\langle [b,T]g_{\wt{P}},h_P \big\rangle+  \big\langle [b,T]h_{\wt{P}},g_P \big\rangle + \int \wt{b}\wt{f}_P } \\
	&\leq \sum_{k=0}^N\sum_{P\in \mathscr{S}_k^2}\babs{ \big\langle [b,T]g_{\wt{P}},h_P \big\rangle} +  \sum_{k=0}^N\sum_{P\in \mathscr{S}_k^2}\babs{\big\langle [b,T]h_{\wt{P}},g_P \big\rangle} +\abs{ \int \wt{b}\wt{f}_{\Sigma}},
	\end{split}
\end{equation}
where we denote $\wt{f}_{\Sigma} = \sum_{k=0}^N\sum_{P\in \mathscr{S}_k^2} \wt{f}_{P}.$

Let us then focus on the first sum on the right-hand side.
The collection $ \mathscr{S}_k^2$ is not necessarily finite and the off-support norm \ref{norm:2} only controls finite sums. Hence we write
\[
\mathscr{S}_{k,j}^2 = \mathscr{S}_k^2\cap \mathscr{S}^2_j,\qquad  \mathscr{S}^2 = \bigcup_{j=1}^{\infty}\mathscr{S}_j^2,\qquad \mathscr{S}_{j}^2 \subset \mathscr{S}_{j+1}^2
\]
for some finite collections $\mathscr{S}_{j}^2 \subset\mathscr{S}^2$ and
\begin{align*}
\sum_{k=0}^N\sum_{P\in \mathscr{S}_k^2}\babs{ \big\langle [b,T]g_{\wt{P}},h_P \big\rangle} = \lim_{j\to\infty}\sum_{k=0}^N\sum_{P\in\mathscr{S}_{k,j}^2}\babs{ \big\langle [b,T]g_{\wt{P}},h_P \big\rangle} 
\end{align*}

Notice that the term $ \big\langle [b,T]g_{\wt{P}},h_P \big\rangle$ is bilinear and hence that we may replace $g_{\wt{P}}$ with $\alpha_P g_{\wt{P}}$ and $h_P$ with $\alpha^{-1}_Ph_P$, for any $\alpha_P\not=0.$
We choose $\alpha_P = \Norm{f_P}{\infty}^{\frac{r_2'}{p_2}}$ and estimate
\begin{equation}\label{disp1}
	\begin{split}
	&\sum_{k=0}^N\sum_{P\in\mathscr{S}_{k,j}^2}\babs{ \big\langle [b,T]g_{\wt{P}},h_P \big\rangle} =
	\sum_{k=0}^N\sum_{P\in\mathscr{S}_{k,j}^2}\Babs{ \Big\langle \left[   b,T\right] \Norm{f_P}{\infty}^{\frac{r_2'}{p_2}}1_{\wt{I\times P}} , \Norm{f_P}{\infty}^{-\frac{r_2'}{p_2}}h_P\Big\rangle } \\
	&\lesssim  \mathcal{O}_{p,q}^{\Sigma}(b;K)\BNorm{\sum_{k=0}^N\sum_{P\in \mathscr{S}_{k,j}^2} \Norm{f_P}{\infty}^{\frac{r_2'}{p_2}}1_{\wt{I\times P}}}{ L^{p_1}_{x_1}L^{p_2}_{x_2}}  \BNorm{\sum_{k=0}^N\sum_{P\in \mathscr{S}_{k,j}^2}\Norm{f_P}{\infty}^{\frac{r_2'}{q_2'}}1_{I\times P}}{ L^{p_1'}_{x_1}L^{q_2'}_{x_2}},
	\end{split}
\end{equation}
in the last step we used the Definition \ref{norm:2} of $\mathcal{O}^{\Sigma}_{p,q},$ the estimate $\Norm{h_P^1}{\infty} \lesssim \Norm{f_P}{\infty}$ and 
the identity  $1 - \frac{r_2'}{p_2} = \frac{r_2'}{q_2'}.$ 
By Lemma \ref{lem:prstop} we have 
$$
\sum_{k=0}^N\sum_{P\in \mathscr{S}_{k,j}^2} \bNorm{f_P}{\infty}^{\frac{r_2'}{q_2'}}1_{I\times P} \lesssim 1_I\otimes (\mathsf{M}f)^{r_2'/q_2'}
$$ and this is enough to control the right-most term of the display \eqref{disp1}. To obtain the similar estimate for the other term we argue as follows. With the rectangle $I\times P$ fixed, write the reflected rectangle as $\wt{I\times P} = \wt{I}_{I\times P}\times\wt{P}_{I\times P}.$  Then, by Proposition \ref{bootstrap} we have
\[
\dist(\wt{I}_{I\times P},I) \sim \diam(I),\qquad \dist(\wt{P}_{I\times P},P) \sim \diam(P),
\] 
and hence, there exists some absolute bounded positive constant $C$ such that  $C\wt{I\times P} \supset I\times P \supset I\times E_P.$
This shows that the collection $\big\{ C\wt{I\times P}: P\in\mathscr{S}\big\}$ of rectangles is sparse with the major subsets $I\times E_P.$ Hence, we have
\allowdisplaybreaks \begin{align*}
&\BNorm{\sum_{k=0}^N\sum_{P\in \mathscr{S}^2} \bNorm{f_P}{\infty}^{\frac{r_2'}{p_2}}1_{\wt{I\times P}}}{L^{p_1}_{x_1}L^{p_2}_{x_2}} \leq \BNorm{\sum_{k=0}^N\sum_{P\in \mathscr{S}^2} \bNorm{f_P}{\infty}^{\frac{r_2'}{p_2}}1_{C\wt{I\times P}}}{L^{p_1}_{x_1}L^{p_2}_{x_2}} \\
&\overset{*}{\lesssim} \BNorm{\sum_{k=0}^N\sum_{P\in \mathscr{S}^2} \bNorm{f_P}{\infty}^{\frac{r_2'}{p_2}}1_{I\times E_P}}{L^{p_1}_{x_1}L^{p_2}_{x_2}} \leq \BNorm{\sum_{k=0}^N\sum_{P\in \mathscr{S}^2} \bNorm{f_P}{\infty}^{\frac{r_2'}{p_2}}1_{I\times P}}{L^{p_1}_{x_1}L^{p_2}_{x_2}},
\end{align*}
where the estimate marked with $*$ can be seen by dualizing and using sparseness, indeed, we have with any function such that $\Norm{ g}{L^{p_1'}_{x_1}L^{p_2'}_{x_2}} \leq 1,$ with any constants $a_j,$ and with any sparse collection $\{R_j,E_{R_j}\}_j$ of rectangles, that
\begin{align*}
	\int \sum_{j} a_j 1_{R_j}g &= \sum_{j} a_j \ave{g}_{R_j}\abs{R_j} \lesssim \sum_{j} \abs{a_j} \ave{\abs{g}}_{R_j}\abs{E_{R_j}} \leq \int \mathsf{M_S}{g}\sum_{j}\abs{a_j}1_{E_{R_j}} \\
	&\leq  \Norm{ \sum_{j}\abs{a_j}1_{E_{R_j}}}{L^{p_1}_{x_1}L^{p_2}_{x_2}}  \Norm{ \mathsf{M_S}{g}}{L^{p_1'}_{x_1}L^{p_2'}_{x_2}} \lesssim \Norm{ \sum_{j}\abs{a_j}1_{E_{R_j}}}{L^{p_1}_{x_1}L^{p_2}_{x_2}},
\end{align*}
where $\mathsf{M_S}$ is the bi-parameter strong maximal function.  Hence, we have again reduced to the pointwise estimate $\sum_{k=0}^N\sum_{P\in \mathscr{S}^2} \bNorm{f_P}{\infty}^{\frac{r_2'}{p_2}}1_{I\times P} \lesssim 1_I\otimes (\mathsf{M}f)^{\frac{r_2'}{p_2}}$ true by Lemma \ref{lem:prstop}. The same estimates also holds for the other term on the line \eqref{line1}.
Putting the above together, we have now shown that
\begin{equation}\label{demonstrate1}
	\begin{split}
		&\sum_{k=0}^N\sum_{P\in \mathscr{S}_k^2}\babs{ \big\langle [b,T]g_{\wt{P}},h_P \big\rangle} +  \sum_{k=0}^N\sum_{P\in \mathscr{S}_k^2}\babs{\big\langle [b,T]h_{\wt{P}},g_P \big\rangle}\\ 
		 &\lesssim  \mathcal{O}_{p,q}^{\Sigma}(b;K)\bNorm{ 1_I\otimes (\mathsf{M}f)^{\frac{r_2'}{p_2}}}{ L^{p_1}_{x_1}L^{p_2}_{x_2}}  \bNorm{1_I\otimes (\mathsf{M}f)^{\frac{r_2'}{q_2'}}}{ L^{p_1'}_{x_1}L^{q_2'}_{x_2}} \\
		&\lesssim   \mathcal{O}_{p,q}^{\Sigma}(b;K)\abs{I}^{1/p_1+1/q_1'} = \mathcal{O}_{p,q}^{\Sigma}(b;K)\abs{I},
	\end{split}
\end{equation}
where we used the boundedness of the maximal function and that $\Norm{f}{L^{r_2'}(\R^{d_2})} \leq 1.$ 
The estimate \eqref{demonstrate1} is uniform in $j$ and hence from \eqref{line1} we find that 
\allowdisplaybreaks \begin{align}\label{estimate1}
\abs{\int_R \wt{b}f} \lesssim  \mathcal{O}_{p,q}^{\Sigma}(b;K)\abs{I}+  \abs{\int_R \wt{b}\wt{f}_{\Sigma}}.
\end{align}
To have control over the error term, we use Proposition \ref{wf:2par} and Lemma \ref{lem:prstop} to find
\begin{align}\label{errorcontrol}
	\abs{\wt{f}_{\Sigma}} &\leq \sum_{k=0}^N\sum_{P\in \mathscr{S}_k^2}\abs{\wt{f}_P} \lesssim \omega(\frac{1}{A}) 1_I\otimes\sum_{k=0}^N\sum_{P\in \mathscr{S}_k^2}\Norm{f_P}{\infty}1_P \lesssim  \omega(\frac{1}{A})1_I\otimes \mathsf{M}f.
\end{align}
%Especially the function $\abs{\wt{f}_{\Sigma}}$ satisfies the conditions \eqref{hugo} with the additional decay $\omega(1/A)$
By \eqref{errorcontrol} and Hölder's inequality we have 
\begin{align*}
	\abs{\int_R \wt{b}\wt{f}_{\Sigma}} &\leq \int_I \Norm{\tilde{b}}{L^{r_2}_{x_2}(J)}\Norm{\wt{f}_{\Sigma}}{L^{r_2'}_{x_2}(J)} \lesssim \int_I \Norm{\tilde{b}}{L^{r_2}_{x_2}(J)} \omega(\frac{1}{A})\Norm{\mathsf{M}f}{L^{r_2'}_{x_2}(J)} \lesssim  \omega(\frac{1}{A})\int_I \Norm{\tilde{b}}{L^{r_2}_{x_2}(J)}
\end{align*}
and hence continuing from the line \eqref{estimate1} after dividing by $\abs{I}$ that
\begin{align*}\label{supover}
\abs{\fint_I\int_J \wt{b}f} \lesssim  \mathcal{O}_{p,q}^{\Sigma}(b;K) +  \omega(\frac{1}{A})\fint_I \Norm{\tilde{b}}{L^{r_2}_{x_2}(J)}.
\end{align*}
Hence, by having $I\to\{x_1\},$ the Lebesgue differentiation theorem shows that
\[
\abs{\int_J\wt{b}(x_1,x_2)f(x_2)\ud x_2} \lesssim  \mathcal{O}_{p,q}^{\Sigma}(b;K) +  \omega(\frac{1}{A}) \Norm{\tilde{b}(x_1,x_2)}{L^{r_2}_{x_2}(J)}.
\]
Since $\wt{b}(x_1,\cdot) \in L^{r_2}(J)$ we have
\[
\sup_{\eqref{hugo}}\abs{ \int_{ J}\wt{b}(x_1,x_2)f(x_2)\ud x_2} = \Norm{\tilde{b}(x_1,x_2)}{L^{r_2}_{x_2}(J)},
\]
where the supremum is taken over all such $f$ as were considered on the line \eqref{hugo}. 
Consequently, we have shown that 
\[
\Norm{\tilde{b}(x_1,x_2)}{L^{r_2}_{x_2}(J)} \lesssim  \mathcal{O}_{p,q}^{\Sigma}(b;K) +  \omega(\frac{1}{A}) \Norm{\tilde{b}(x_1,x_2)}{L^{r_2}_{x_2}(J)}.
\]
The term shared on both sides of the estimate is finite almost everywhere and hence by absorbing the common term to the left-hand side we conclude with the left-most estimate of \eqref{smurffi}.

The estimate on the middle was already discussed earlier in section \ref{norm:2} and it remains to show the right-most estimate. As the commutator is unchanged modulo constants, we find that
\begin{align*}
	\Norm{[b,T]f}{ L^{p_1}_{x_1}L^{q_2}_{x_2}}  = 	\Norm{[b-c,T]f}{ L^{p_1}_{x_1}L^{q_2}_{x_2}} \leq 	\Norm{(b-c)Tf}{ L^{p_1}_{x_1}L^{q_2}_{x_2}} + 	\Norm{T(b-c)f}{ L^{p_1}_{x_1}L^{q_2}_{x_2}}. 
\end{align*}
From  here, by the mixed norm estimates of $T,$ it is enough to estimate
\begin{align*}
	\Norm{(b-c)f}{ L^{p_1}_{x_1}L^{q_2}_{x_2}} \leq \bNorm{\Norm{b-c}{L^{r_2}_{x_2}}\Norm{f}{L^{p_2}_{x_2}}}{L^{p_1}_{x_1}} \leq \Norm{b-c}{L^{\infty}_{x_1}L^{r_2}_{x_2}}\Norm{f}{L^{p_1}_{x_1}L^{p_2}_{x_2}},
\end{align*}
where we used that $1/q_2 = 1/r_2+1/p_2.$
Taking the infimum over all $c\in\C$ shows the claim.
\end{proof}

\begin{prop}\label{prop:infty2r1} Let $p_2=q_2$ and $q_1<p_1$ and assume that $b\in L^1_{\loc,x_2}L^{r_1}_{\loc,x_1}.$ Then, there holds that 
	$$
	\inf_{c\in\C}\Norm{b-c}{L^{\infty}_{x_2} L^{r_1}_{x_1}} \lesssim \mathcal{O}_{p,q}^{\Sigma}(b;K) \leq 	\Norm{[b,T]}{L^{p_1}_{x_1}L^{p_2}_{x_2}\to L^{q_1}_{x_1}L^{p_2}_{x_2}} \lesssim \inf_{c\in\C}\Norm{b-c}{L^{r_1}_{x_1}L^{\infty}_{x_2} }.
	$$	
\end{prop} 
\begin{proof} The left-most estimate is completely symmetric with the proof of Proposition \ref{prop:infty1r2} and the estimate on the middle is immediate by Hölder's inequality. The right-most estimate follows by the invariance of the commutator modulo additive constants, the mixed norm estimates of $T,$ and Hölders inequality.
\end{proof}

\subsection{The case $p_1>q_1$ and $p_2>q_2$} In this case, again, it follows immediately by Hölder's inequality, the invariance of the commutator modulo additive constants, and the mixed norm estimates of $T$, that
\[
\mathcal{O}^{\Sigma}_{p,q}(b;K)\leq \Norm{[b,T]}{L^{p_1}_{x_1}L^{p_2}_{x_2}\to L^{q_1}_{x_1}L^{q_2}_{x_2}} \leq  \inf_{c\in\C}\Norm{b-c}{L^{r_1}_{x_1}L^{r_2}_{x_2} }.
\] 
Then, we would like to prove a lower bound for $\mathcal{O}^{\Sigma}_{p,q}(b;K)$ that gets as close to $\inf_{c\in\C}\Norm{b-c}{L^{r_1}_{x_1}L^{r_2}_{x_2}}$ as possible. Let us first discuss the non-mixed case, where we have a full characterization.
\begin{prop}\label{prop:r1r2diag} Let $p_1=p_2 > q_1=q_2,$ define $1/r = 1/q_1-1/p_1,$ and assume that $b\in L^1_{\loc}.$ Then, there holds that
	$$
	\inf_{c\in\C}\Norm{b-c}{L^r(\R^{d})}\sim  \mathcal{O}^{\Sigma}_{p,q}(b;K) \sim \Norm{[b,T]}{L^p(\R^d)\to L^q(\R^d)}.
	$$
\end{prop}
\begin{proof} 
	The following oscillatory characterization is recorded e.g. as Proposition 3.2. in \cite{AHLMO}. Let $r\in(1,\infty),$ then there holds that 
	\begin{align}\label{sparseosc}
	\inf_{c\in\C}\Norm{b-c}{L^r(\R^d)} \sim \sup_{\mathscr{S}}\Big\{ \sum_{Q\in\mathscr{S}}\lambda_Q\abs{Q}\osc(b,Q)\colon \mathscr{S} \mbox{ is } 1/2\mbox{-sparse}, \sum_{Q\in\mathscr{S}}\abs{Q}\lambda_Q^{r'}\leq 1\Big\},
	\end{align}
	where the sparse collections $\mathscr{S}$ consist of cubes of $\R^{d}.$
	Now, fix any sparse collection $\mathscr{S}$ as in the supremum. Then, identically as in the proof of Proposition \ref{prop:r1r2}, we can bound 
	$$
	\sum_{Q\in\mathscr{S}}\lambda_Q\abs{Q}\osc(b,Q) \lesssim \mathcal{O}^{\Sigma}_{p,q}(b;K).
	$$ The remaining bounds $\mathcal{O}^{\Sigma}_{p,q}(b;K) \lesssim \Norm{[b,T]}{L^p(\R^d)\to L^q(\R^d)} \lesssim \inf_{c\in\C}\Norm{b-c}{L^r(\R^d)}$ were already discussed above in the mixed case.
\end{proof}

In the mixed case we are unable to prove the desired lower bound and what we have is the following
\begin{prop}\label{prop:r1r2} Let $p_1>q_1,p_2>q_2$ and let $b\in L^1_{\loc}.$ Let $\mathscr{S}^i$ denote a  $1/2$-sparse collection on $\R^{d_i}$ with associated coefficients $\{\lambda_{I_i}\}$ such that $\sum_{I_i\in\mathscr{S}^i}\lambda_{I_i}^{r_i'}\abs{I_i} \leq 1.$
	Then,there holds that
	\allowdisplaybreaks \begin{align*}
	\sup_{\mathscr{S}^1,\mathscr{S}^2}\Big[  \sum_{I_1\in\mathscr{S}^1}\sum_{I_2\in\mathscr{S}^2} \lambda_{I_1}\lambda_{I_2}\abs{I_1}\abs{I_2}\osc(b,I_1\times I_2) \Big]  
	\lesssim \mathcal{O}^{\Sigma}_{p,q}(b;K).
	\end{align*}
\end{prop}
	Technically this limitation is due to the failure of finding any useful rectangular sparse oscillatory characterization of the mixed space $L^s_{x_1}L^t_{x_2},$ when $s\not = t,$ that would correspond with that of the one on the line \eqref{sparseosc}.
\begin{proof}[Proof of Proposition \ref{prop:r1r2}] Without loss of generality we may assume that the collections $\mathscr{S}^i$ are finite. First, by Proposition \ref{prop:oscform} we have 
\begin{align}\label{ddd}
	\abs{I_1}\abs{I_2}\osc(b,I_1\times I_2) \lesssim  \abs{\langle [b,T]g_{I_1\times I_2}^1, h_{I_1\times I_2}^1,\rangle} +\abs{\langle [b,T]h_{I_1\times I_2}^2,g_{I_1\times I_2}^2\rangle}
\end{align}
		where we write $h_{I_1\times I_2}^i,g_{I_1\times I_2}^i$ for the functions $g_i,h_i$. Also, let $R(h_{I_1\times I_2}^i)$ and $R(g_{I_1\times I_2}^i)$ stand respectively for the rectangles on which $h_{I_1\times I_2}^i$ and $g_{I_1\times I_2}^i$ are supported.
		Then, by \eqref{ddd}, the relation $1/r_i = 1/q_i-1/p_i,$ and the Definition \ref{norm:2} of the off-support norm, we estimate 
	\allowdisplaybreaks \begin{align*}
	&\sum_{I_1\in\mathscr{S}^1}\sum_{I_2\in\mathscr{S}^2} \lambda_{I_1}\lambda_{I_2}\abs{I_1}\abs{I_2}\osc(b,I_1\times I_2) \\  
	&\lesssim  \sum_{I_1\in\mathscr{S}^1}\sum_{I_2\in\mathscr{S}^2} \lambda_{I_1}\lambda_{I_2}\abs{\langle [b,T]g_{I_1\times I_2}^1, h_{I_1\times I_2}^1,\rangle} + \sum_{I_1\in\mathscr{S}^1}\sum_{I_2\in\mathscr{S}^2}\lambda_{I_1}\lambda_{I_2}\abs{\langle [b,T]h_{I_1\times I_2}^2,g_{I_1\times I_2}^2\rangle}
	\end{align*}
	and let us estimate the sums as 
	\begin{align*}
	&\lesssim \sum_{I_1\in\mathscr{S}^1}\sum_{I_2\in\mathscr{S}^2} \lambda_{I_1}\lambda_{I_2}\abs{\langle [b,T]h_{I_1\times I_2}^i,g_{I_1\times I_2}^i\rangle}  \\
	&=\sum_{i=1,2}\sum_{I_1\in\mathscr{S}^1}\sum_{I_2\in\mathscr{S}^2} \abs{\langle [b,T] (\lambda_{I_1}^{r_1'/p_1}\lambda_{I_2}^{r_2'/p_2}h_{I_1\times I_2}^i),\lambda_{I_1}^{r_1'/q_1'}\lambda_{I_2}^{r_2'/q_2'}g_{I_1\times I_2}^i\rangle} \\
	&\leq  \mathcal{O}^{\Sigma}_{p,q}(b;K) \sum_{i=1,2} \Norm{\sum_{I_1\in\mathscr{S}^1}\sum_{I_2\in\mathscr{S}^2}\lambda_{I_1}^{r_1'/p_1}\lambda_{I_2}^{r_2'/p_2}1_{R(h_{I_1\times I_2}^i)}}{L^{p_1}_{x_1}L^{p_2}_{x_2}} \\
	&\qquad\qquad\times \Norm{\sum_{I_1\in\mathscr{S}^1}\sum_{I_2\in\mathscr{S}^2} \lambda_{I_1}^{r_1'/q_1'}\lambda_{I_2}^{r_2'/q_2'} 1_{R(g_{I_1\times I_2}^i)} }{L^{q_1'}_{x_1}L^{q_2'}_{x_2}}.
	\end{align*}
	Using that the coefficients are of product form, we can then, for example, estimate one of the terms as
	\begin{equation*}
		\begin{split}
			&\Norm{\sum_{I_1\in\mathscr{S}^1}\sum_{I_2\in\mathscr{S}^2} \lambda_{I_1}^{r_1'/q_1'}\lambda_{I_2}^{r_2'/q_2'} 1_{R(g_{I_1\times I_2}^2)} }{L^{q_1'}_{x_1}L^{q_2'}_{x_2}} = \Norm{\sum_{I_1\in\mathscr{S}^1}\sum_{I_2\in\mathscr{S}^2} \lambda_{I_1}^{r_1'/q_1'}\lambda_{I_2}^{r_2'/q_2'} 1_{I_1\times I_2} }{L^{q_1'}_{x_1}L^{q_2'}_{x_2}} \\
		&=  \Norm{\sum_{I_1\in\mathscr{S}^1}\lambda_{I_1}^{r_1'/q_1'} 1_{I_1} }{L^{q_1'}_{x_1}} \Norm{\sum_{I_2\in\mathscr{S}^2}\lambda_{I_2}^{r_2'/q_2'} 1_{ I_2} }{L^{q_2'}_{x_2}} \lesssim 1,
		\end{split}
	\end{equation*} 
	where in the last step we used the sparseness of the collections $\mathscr{S}^i$ and the assumed bounds $\sum_{I_i\in\mathscr{S}^i}\lambda_{I_i}^{r_i'}\abs{I_i} \leq 1.$ The remaining three terms are estimated in the same fashion, basically repeating the arguments that we already went through in the proof of Proposition \ref{prop:infty1r2}.
\end{proof}

%
%\begin{rem}\label{rem:weird} Let $\mathfrak{F}$ stand for any collection of collections of rectangles, e.g. the collection of all $\frac{1}{2}$-sparse families of cubes, and denote
%	\[
%	\mathbb{O}_{\mathfrak{F}}(b) =  \sup\Big\{ \sum_{R\in\mathscr{S}}\lambda_R\abs{R}\osc(b,R)\colon \mathscr{S}\in\mathfrak{F}, \sum_{R\in\mathscr{S}}\abs{R}\lambda_Q^{r'}\leq 1\Big\}.
%	\] 
%	Let $\mathfrak{F}_{Q}$ stand for the collection of all $\frac{1}{2}$-sparse collections of cubes, and let
%	$\mathfrak{F}_{R}$ stand for the collection of all $\frac{1}{2}$-sparse collections of rectangles.
%	Then, a straightforward adaptation of Proposition 3.2. in \cite{AHLMO} would allows us to actually show that 
%	\[
%		\mathbb{O}_{\mathfrak{F}_{R}}(b) \lesssim \inf_{c\in\C}\Norm{b-c}{L^r(\R^d)}  \lesssim \mathbb{O}_{\mathfrak{F}_{Q}}(b).
%	\]
%	This demonstrates that in \eqref{sparseosc} it is irrelevant whether we control all sparse collections of cubes or all  sparse collections of rectangles.
%\end{rem}

\section{Upper bound for the case $p_1=q_1, p_2<q_2$}\label{section:alphahöldercases1} 
 We are now left with two cases and we first deal with this one. We will use the representation of bi-parameter CZO's as dyadic model operators; this is maybe surprising as the corresponding lower bound obtained in Proposition \ref{prop:lb:c1a2} seems simple and should perhaps yield an easier proof.
We will prove 
\begin{prop}\label{prop:c1a2T} Let $p_1=q_1$ and $p_2<q_2,$ let $b(x_1,\cdot) \in \dot C^{0,\alpha_2}(\R^{d_2})$ and $b(\cdot,x_2)=  constant.$ Then, we have 
	\[
	\Norm{[b,T]f}{L^{p_1}_{x_1}L^{q_2}_{x_2}} \lesssim 	\Norm{b(x_1,\cdot)}{\dot C^{0,\alpha}_{x_2}(\R^{d_2})} \Norm{f}{L^{p_1}_{x_1}L^{p_2}_{x_2}}.
	\]
\end{prop}
The dyadic representation theorem of bi-parameter CZO's of Martikainen \cite{Ma1} is the following
\begin{thm}\label{thm:birep} Given a bi-parameter CZO, it can be written as an expectation
	\[
	\langle Tf,g\rangle = C_T\mathbb{E}_{\omega_1}\mathbb{E}_{\omega_2} \sum_{\substack{i = (i_1,i_2)\in \N^2 \\ j=(j_1,j_2)\in \N^2}} 2^{-\max(i_1,i_2)\frac{\delta}{2}} 2^{-\max(j_1,j_2)\frac{\delta}{2}}\big\langle S^{i,j}_{\calD^1_{\omega_1},\calD^2_{\omega_2}}f,g\big\rangle,
	\]
	where $ S^{i,j}_{\calD^1_{\omega_1},\calD^2_{\omega_2}}$ are bi-parameter dyadic model operators (detailed below) associated to the randomized dyadic grids $\calD^1_{\omega_1}$ and $\calD^2_{\omega_2}.$
\end{thm}

By Theorem \ref{thm:birep} to have estimates for $[b,T],$ it is enough to have them for $[b,S^{i,j}],$ where $S^{i,j}$ is a dyadic model operator, and with constants of at most polynomial growth in the parameters $i, j,$ namely, it is enough to prove the following
\begin{prop}\label{prop:c1a2} Let $p_1=q_1$ and $p_2<q_2,$ let $b(x_1,\cdot) \in \dot C^{0,\alpha_2}(\R^{d_2})$ and $b(\cdot,x_2)=  constant.$ Then, we have 
	\[
	\bNorm{[b,S^{i,j}]f}{L^{p_1}_{x_1}L^{q_2}_{x_2}} \lesssim \Norm{b(x_1,\cdot)}{\dot C^{0,\alpha_2}(\R^{d_2})}\Norm{f}{L^{p_1}_{x_1}L^{p_2}_{x_2}}
	\]
	with an implied constant of at most polynomial growth in $i,j.$
\end{prop}
We have that $S^{i,j}$ is either a shift, a partial paraproduct or a full paraproduct, to detail each of which we first recall few basic facts about martingale differences and Haar functions, the reader familiar with these may skip to Section \ref{decomp}.

\subsubsection{Haar functions, basic facts}

Given a dyadic grid $\calD$ and a cube $I\in\calD$ the martingale difference on $I$ is 
$$
\Delta_If = \sum_{P\in\ch(I)}\big(\ave{f}_P-\ave{f}_I\big)1_P.
$$ 
These are naturally useful as $f = \sum_{I\in\calD}\Delta_If,$
where each element is nicely localized and has zero mean.
For a given interval $I = I_l\cup I_r\subset\R,$ with a left- and a right half, respectfully the cancellative and non-cancellative Haar functions supported $I$ are
$$
h_I = \frac{1_{I_l}-1_{I_r}}{\abs{I}^{1/2}},\qquad h_I^0 = \frac{1_I}{\abs{I}^{1/2}}.
$$
Given a rectangle $I = I_1\times\dots\times I_d\subset\R^d,$ the Haar functions on $I$ are
$$
h_I = \otimes_{i=1}^d \wt h_{I_i},\qquad \wt{h}_{I_i}\in\{h_{I_i},h_{I_i}^0\}
$$
on the condition that at least one component is a cancellative Haar function.
Hence, all in all, there are $2^d-1$ Haar functions on any given rectangle of dimension $d,$ along with the non-cancellative Haar function $1_I/\abs{I}^{1/2}.$ 
It is a basic fact that
\begin{align}\label{block}
	\Delta_If = \sum_{i=1}^{2^d-1}\langle f,h_I^i\rangle h_I^i
\end{align}
where we enumerate all $2^{d}-1$ cancellative Haar functions on the rectangle $I.$ 
Hence, when proving upper bounds we just write $h_I = h_I^i$ for a generic cancellative Haar function on $I$ and it is customary to ignore the $i=1,\dots,2^{d}-1$ summation in \eqref{block}.

Fix a rectangle $R=I\times J\subset\R^{d_1}\times\R^{d_2}.$
Fully cancellative Haar functions of product form are the tensor products $h_R=h_I\otimes h_J,$ where both $h_I,h_J$ are cancellative Haar functions respectfully on $I$ and $J.$ Then, simply,
\begin{align*}
\Delta_{I\times J}f = \Delta_I\big(\Delta_Jf\big) = \sum_{i=1}^{2^{d_1}-1}\sum_{j=1}^{2^{d_2}-1}\bave{\bave{f,h_J^j}h_J^j,h_I^i}h_I^i = \sum_{i=1}^{2^{d_1}-1}\sum_{j=1}^{2^{d_2}-1}\bave{f,h_I^i\otimes h_J^j}h_I^i\otimes h_J^j, 
\end{align*}
and again each Haar $h_R = h_I^i\otimes h_J^j$ carries enough cancellation for boundedness of bi-parameter square functions etc.

\subsection{Model operators}

A pair of intervals we denote $(I)= (I_1,I_2)$ and with  $I^k  = I^{(k)}= Q$ we mean that $I,Q\in\calD$, $I\subset Q$ and $\ell(I) = 2^{-k}\ell(Q).$ 
Now, the bi-parameter dyadic model operators of Theorem \ref{thm:birep} have the generic form
\[
\big\langle S^{i,j}f,g\big\rangle = \sum_{\substack{K\in\calD^1 \\ I_1^{i_1} = I_2^{i_2} = K }}\sum_{\substack{V\in\calD^2 \\ J_1^{j_1} = J_2^{j_2} = V }}\alpha_{(I)(J)KV}\langle f, \widetilde{h}_{I_1\times J_1}\rangle \langle g,  \widetilde{h}_{I_2\times J_2}\rangle,
\]
where the coefficients $\alpha_{(I)(J)KV}$ have sizes according to which dyadic model operator we have:
There are in total three different kinds of model operators that appear in \ref{thm:birep}.
\subsubsection{Shifts}
We have 
$$
\langle f, \widetilde{h}_{I_1\times J_1}\rangle \langle g,  \widetilde{h}_{I_2\times J_2}\rangle = \langle f, h_{I_1\times J_1}\rangle \langle g,  h_{I_2\times J_2}\rangle
$$
where each of the Haar functions is cancellative
and the coefficients have the size 
$$
\abs{\alpha_{(I)(J)KV}} \lesssim \frac{(\abs{I_1}\abs{I_2}\abs{J_1}\abs{J_2})^{1/2}}{\abs{K\times V}}.
$$ 

\subsubsection{Partial paraproducts}
We have $i_1=i_2 = 0$ and
\[
\langle f, \widetilde{h}_{I_1\times J_1}\rangle \langle g,  \widetilde{h}_{I_2\times J_2}\rangle = \langle f, \frac{1_K}{\abs{K}}\otimes h_{J_1}\rangle \langle g,  h_{K}\otimes h_{J_2}\rangle,
\]
or the symmetric case,
$$
\langle f, \widetilde{h}_{I_1\times J_1}\rangle \langle g,  \widetilde{h}_{I_2\times J_2}\rangle = \langle f, h_K\otimes h_{J_1}\rangle \langle g,  \frac{1_K}{\abs{K}}\otimes h_{J_2}\rangle,
$$ 
and in both cases the coefficients have the size
\begin{align*}
\Norm{	(\alpha_{(J)KV})_K}{\BMO_2(\R^{d_1})} &= \sup_{K_0\in\calD^1}\frac{1}{\abs{K_0}^{1/2}}\Norm{\big(\sum_{\substack{K\in\calD \\ K\subset K_0}}\abs{\alpha_{(J)KV}}^2\frac{1_K}{\abs{K}}\big)^{1/2}}{L^2(\R^{d_1})} \\
&\lesssim \frac{\abs{J_1}^{1/2}\abs{J_2}^{1/2}}{\abs{V}}.
\end{align*}

There is also the other symmetry of $j_1=j_2 = 0,$ and then
\[
\langle f, \widetilde{h}_{I_1\times J_1}\rangle \langle g,  \widetilde{h}_{I_2\times J_2}\rangle = \langle f,h_{I_1}\otimes \frac{1_V}{\abs{V}}\rangle \langle g,  h_{I_2}\otimes h_{V}\rangle,
\]
and its symmetric case
\[
\langle f, \widetilde{h}_{I_1\times J_1}\rangle \langle g,  \widetilde{h}_{I_2\times J_2}\rangle =  \langle f,  h_{I_1}\otimes h_{V}\rangle\langle g,h_{I_2}\otimes \frac{1_V}{\abs{V}}\rangle,
\]
and in both of these two cases the coefficients have the size
\begin{align*}
\Norm{	(\alpha_{(I)KV})_V}{\BMO_2(\R^{d_2})} &= \sup_{V_0\in\calD^1}\frac{1}{\abs{V_0}^{1/2}}\Norm{\big(\sum_{\substack{V\in\calD \\ V\subset V_0}}\abs{\alpha_{(I)KV}}^2\frac{1_V}{\abs{V}}\big)^{1/2}}{L^2(\R^{d_2})} \\
&\lesssim \frac{\abs{I_1}^{1/2}\abs{I_2}^{1/2}}{\abs{K}}.
\end{align*}

\subsubsection{Full paraproducts}
We have $i_1=i_2=j_1=j_2 = 0$ and
\[
\langle f, \widetilde{h}_{I_1\times J_1}\rangle \langle g,  \widetilde{h}_{I_2\times J_2}\rangle = \ave{f}_{K\times V}  \langle g,  h_{K}\otimes h_{V}\rangle
\]
or the symmetric case
\[
\langle f, \widetilde{h}_{I_1\times J_1}\rangle \langle g,  \widetilde{h}_{I_2\times J_2}\rangle =  \langle f,  h_{K}\otimes h_{V}\rangle\ave{g}_{K\times V}, 
\]
or we have the other symmetry
\[
\langle f, \widetilde{h}_{I_1\times J_1}\rangle \langle g,  \widetilde{h}_{I_2\times J_2}\rangle = \langle f, h_K\otimes\frac{1_V}{\abs{V}} \rangle \langle g,  \frac{1_K}{\abs{K}}\otimes h_{V}\rangle
\]
and its symmetric case 
\[
\langle f, \widetilde{h}_{I_1\times J_1}\rangle \langle g,  \widetilde{h}_{I_2\times J_2}\rangle = \langle f , \frac{1_K}{\abs{K}}\otimes h_{V}\rangle\langle g, h_K\otimes\frac{1_V}{\abs{V}} \rangle.
\] 
 The boundedness of full paraproducts bootstraps directly to Proposition \ref{blackbox} below and to the boundedness of fractional operators. Hence, we will not record their coefficient size, nonetheless, we mention that the coefficient size is measured by the product BMO space of Chang and Fefferman, see e.g. Section 7 in \cite{AHLMO}.

The following Proposition \ref{blackbox} is e.g. contained as a part of Hyt\"onen-Martikainen-Vuorinen \cite{HMV}.
\begin{prop}\label{blackbox} All the above described dyadic model operators, the shifts, the partial paraproducts and the full paraproducts, are bounded
	\[
	\Norm{ S^{i,j}f}{L^{p_1}_{x_1}L^{p_2}_{x_2}} \lesssim \Norm{f}{L^{p_1}_{x_1}L^{p_2}_{x_2}} 
	\]
	with an implied constant of at most polynomial growth in $i,j.$
\end{prop}

\subsubsection{Decomposition of products}\label{decomp}
Notice that as the function $b$ bears no important information in the first variable, we only need to analyse it carefully in the second parameter, which we do according to the commutator decomposition strategy from \cite{LMV:Bloom}:
\begin{enumerate}[(i)]
	\item  Whenever a product $bf$ (or $bg$) is paired against a cancellative Haar function $h_J$ and $J\in\calD^2$, we expand with respect to the dyadic grid $\calD^2$ as
	\allowdisplaybreaks \begin{align*}
	bf &= \sum_{J\in\calD^2}\Delta_Jb\Delta_Jf + \sum_{J\in\calD^2}\Delta_JbE_Jf+\sum_{J\in\calD^2}E_Jb\Delta_Jf \\
	& = A_1(b,f)+A_2(b,f)+A_3(b,f),
	\end{align*}
	where we denote $E_Jb = \ave{b}_J1_J.$ It should be understood that the operators $A_i$ depend on the fixed dyadic grid $\calD^2$ even though we omit this detail from the notation. Especially, if our model operators $S^{i,j}$ are defined on the grid $\calD^1\times\calD^2$, then we will expand in the grid $\calD^2.$
	\item If $bf$ is averaged in the second parameter, then we add and subtract $\ave{bf}_J1_J,$
	\allowdisplaybreaks \begin{align*}
	bf1_J = (bf - \ave{bf}_J)1_J + \ave{bf}_J1_J.
	\end{align*}
\end{enumerate}
The proof of Proposition \ref{prop:c1a2} splits into several cases of which some are symmetric; as there are too many cases to present here fully, we go through a proof of each representative case for each model operator after which it is clear how to carry through the remaining cases.

The first step is to establish the boundedness for the auxiliary operators.
\begin{prop}\label{prop:frac:para} Let $1<p<q<\infty$ and $\alpha = d(1/p-1/q).$
	Then,
	\allowdisplaybreaks 
	\begin{align}\label{hhh}
	\Norm{A_i(b,f)}{L^q(\R^d)} \lesssim \Norm{b}{\dot C^{0,\alpha}(\R^d)}\Norm{f}{L^{p}(\R^{d})}.
	\end{align}
\end{prop}

\begin{proof} 
	Let us first estimate 
	\begin{equation}\label{simpleX}
		\begin{split}
		\abs{A_i(b,f)} &\leq \sum_{Q\in\calD} \frac{\abs{\langle b, h_Q\rangle}}{\abs{Q}^{1/2}}\ave{\abs{f}}_Q1_Q = \sum_{Q\in\calD} \frac{\abs{\langle b-\ave{b}_Q, h_Q\rangle}}{\abs{Q}^{1/2}}\ave{\abs{f}}_Q1_Q \\
		&\leq \sum_{Q\in\calD} \ave{\abs{b-\ave{b}_Q}}_Q\ave{\abs{f}}_Q1_Q \leq \Norm{b}{\dot C^{0,\alpha}(\R^d)}\sum_{Q\in\calD} \ell(Q)^{\alpha}\ave{\abs{f}}_Q1_Q.
		\end{split}
	\end{equation}
	Then, we show that the positive operator
	\begin{align}\label{aid1}
		\mathsf{A}^{\alpha}_{\calD}f = \sum_{Q\in\calD} \ell(Q)^{\alpha}\ave{\abs{f}}_Q1_Q
	\end{align}
	satisfies the desired bound.
	Fix a top cube $Q_0\in\calD$ and let $\mathscr{S}\subset \calD(Q_0)$ be the stopping time sparse collection inside the cube $Q_0$ as described in the beginning of Section \ref{sect:super}. By the stopping condition and sparseness of $\mathscr{S},$ we estimate
	\begin{align*}
	\Norm{\mathsf{A}^{\alpha}_{\calD_{Q_0}}f}{L^{q}(\R^d)} &= 	\bNorm{  \sum_{P\in\mathscr{S}}\sum_{\Pi Q = P}\ell(Q)^{\alpha}\ave{\abs{f}}_Q1_Q }{L^{q}(\R^d)} \lesssim \bNorm{  \sum_{P\in\mathscr{S}}\ave{\abs{f}}_P\sum_{\substack{Q\in\calD_{Q_0} \\ Q\subset P}}\ell(Q)^{\alpha}1_Q }{L^{q}(\R^d)} \\
	&\leq \bNorm{  \sum_{P\in\mathscr{S}}\ave{\abs{f}}_P \big(\sum_{k=0}^{\infty}2^{-k\alpha}) \ell(P)^{\alpha}1_P }{L^{q}(\R^d)} \lesssim \bNorm{  \sum_{P\in\mathscr{S}}\ell(P)^{\alpha}\ave{\abs{f}}_P1_P}{L^{q}(\R^d)}  \\
	&\overset{*}{\lesssim} \bNorm{  \sum_{P\in\mathscr{S}}\ell(P)^{\alpha}\ave{\abs{f}}_P1_{E_P}}{L^{q}(\R^d)} =\big(\sum_{P\in\mathscr{S}}\int_{E_P}(\ell(P)^{\alpha}\ave{\abs{f}}_P)^q\big)^{1/q} \\
	&\leq \Norm{\mathsf{M}^{\alpha}f}{L^q(\R^d)} \lesssim \Norm{f}{L^p(\R^d)},
	\end{align*}
	where at the estimate marked with $*$ we used the sparseness of $\mathscr{S}$  to get the norm estimate (for details, see the similar estimate in the proof of Proposition \ref{prop:infty1r2}), and where the boundedness of the fractional maximal operator,
	\allowdisplaybreaks \begin{align*}
	\mathsf{M}^{\alpha}f(x) = \sup_{Q}1_Q(x)\ell(Q)^{\alpha}\fint_Q\abs{f},
	\end{align*}
	where the supremum is taken over all cubes $Q\subset\R^d,$ was used.
	 As the demonstrated bound is independent of the choice of the top cube $Q_0,$ we get the boundedness for $\mathsf{A}^{\alpha}_{\calD}$ and hence \eqref{hhh}.
\end{proof} 

%and again, the last estimate follows from the following estimate (we will refer here later on and hence some extra steps are given)
%\begin{equation}\label{simpleX}
%\begin{split}
%  \sup_{S,P\subset Q}\abs{\ave{b}_P-\ave{b}_S} &\leq \sup_{S,P\subset Q}\ave{\abs{b-\ave{b}_P}}_S \leq  \sup_{\substack{S,P\subset Q}} \fint_S\fint_P\abs{b(x)-b(y)}\ud y \ud x \\
%&\leq \Norm{b}{\dot C^{0,\alpha}(\R^d)}\sup_{\substack{S,P\subset Q}} \fint_S\fint_P\abs{x-y}^{\alpha}\ud y \ud x \\
% &\leq  \Norm{b}{\dot C^{0,\alpha}(\R^d)}\ell(Q)^{\alpha}\sup_{\substack{S,P\subset Q}} \fint_S\fint_P\ud y \ud x = \Norm{b}{\dot C^{0,\alpha}(\R^d)}\ell(Q)^{\alpha},
%\end{split}
%\end{equation}
%where the supremum is taken at least over all cubes $S,P$.

We will also have use of the following fractional Fefferman-Stein inequality, recorded e.g. in \cite{CMN2019}.
\begin{lem}\label{lem:frac:FS} Let $1<p< q<\infty$, $\alpha = d(1/p-1/q)<d,$ and $1<r<\infty.$ Then, there holds that
	\allowdisplaybreaks \begin{align*}
	\bNorm{\big( \sum_k(\mathsf{M}^{\alpha}f_k)^r\big)^{1/r}}{L^q(\R^d)} \lesssim_{d,p,q,r} 	\bNorm{\big( \sum_k \abs{f_k}^r\big)^{1/r}}{L^p(\R^d)}.
	\end{align*}
\end{lem}
\begin{rem} Lemma \ref{lem:frac:FS} becomes Fefferman-Stein inequality when $p=q.$
\end{rem}

For the following two lemmas see e.g. \cite{HMV}.
\begin{lem}\label{lem:bound:sf} Let $1<p_1,p_2<\infty.$ Then, there holds that
	\begin{align*}
	\bNorm{\mathsf{S}^1f}{L^{p_1}_{x_1}L^{p_2}_{x_2}} \sim \bNorm{\mathsf{S}^2f}{L^{p_1}_{x_1}L^{p_2}_{x_2}}\sim \bNorm{\mathsf{S}f}{L^{p_1}_{x_1}L^{p_2}_{x_2}} \sim \bNorm{f}{L^{p_1}_{x_1}L^{p_2}_{x_2}},
	\end{align*}
	hold, where 
	\begin{align*}
	\mathsf{S}^if= \Big(\sum_{L\in\calD^i}\abs{\langle f,h_L \rangle}^2\frac{1_L}{\abs{L}}\Big)^{1/2},\quad 	\mathsf{S}f= \Big(\sum_{\substack{I\in\calD^1 \\ J\in\calD^2}}\abs{\langle f,h_I\otimes h_J \rangle}^2\frac{1_{I\times J}}{\abs{I\times J}}\Big)^{1/2}.
	\end{align*}
\end{lem}
%
%We also use the following bi-parameter Fefferman-Stein inequality.
\begin{lem}\label{lem:FSbp} Let $1<s,t,r<\infty.$ Then, there holds that 
	\begin{align*}
		\bNorm{\big(\sum_{k} \mathsf{M}^1\mathsf{M}^2f_k\big)^{1/r}}{L^s_{x_1}L^t_{x_2}} \lesssim_{s,t,r} 	\bNorm{\big(\sum_{k} \abs{f_k}\big)^{1/r}}{L^s_{x_1}L^t_{x_2}}.
	\end{align*}
\end{lem}

\begin{proof}[Proof of Theorem \ref{prop:c1a2}, part 1/3, shifts:] Let $S^{i,j}$ stand for the model operator 
	\begin{align}\label{asd}
		\big\langle S^{i,j}f,g\big\rangle = \sum_{\substack{K\in\calD^1 \\ I_1^{i_1} = I_2^{i_2} = K }}\sum_{\substack{V\in\calD^2 \\ J_1^{j_1} = J_2^{j_2} = V }}\alpha_{(I)(J)KV}\langle f, h_{I_1\times J_1}\rangle \langle g,  h_{I_2\times J_2}\rangle.
	\end{align}
	By the above described decomposition strategy, we find that the summand (without the scaling factor $\alpha_{(I)(J)KV}$ in front) in \eqref{asd} writes out as
\allowdisplaybreaks \begin{align*}
 	&\big[ \langle f,h_{I_1}\otimes h_{J_1}\rangle \langle bg, h_{I_2}\otimes h_{J_2}\rangle -  \langle bf,h_{I_1}\otimes h_{J_1}\rangle \langle g, h_{I_2}\otimes h_{J_2}\rangle  \big] \\
 	&= \sum_{i=1,2} \langle f,h_{I_1}\otimes h_{J_1}\rangle \langle A_{i}(b,g), h_{I_2}\otimes h_{J_2}\rangle - \sum_{i=1,2}  \langle A_i(b,f),h_{I_1}\otimes h_{J_1}\rangle \langle g, h_{I_2}\otimes h_{J_2}\rangle \\
 	&\qquad\qquad+  \Big[\langle f,h_{I_1}\otimes h_{J_1}\rangle \langle A_3(b,g), h_{I_2}\otimes h_{J_2}\rangle -  \langle A_3(b,f),h_{I_1}\otimes h_{J_1}\rangle \langle g, h_{I_2}\otimes h_{J_2}\rangle\Big].
\end{align*}
The terms with the first four summands are bounded by the mixed norm estimates of bi-parameter model operators and Proposition \ref{prop:frac:para}. Indeed, for the first two terms we use that $A_i(b,\cdot):L^{q'_2}_{x_2}\to L^{p'_2}_{x_2}$ boundedly,
and for the following two terms directly Proposition \ref{prop:frac:para}.
For the bracketed difference on the last line we utilise the cancellation of the commutator, hence writing it out as
\allowdisplaybreaks
\begin{equation}\label{convention}
	\begin{split}
	 &\langle f,h_{I_1}\otimes h_{J_1}\rangle \ave{b}_{J_2}\langle g, h_{I_2}\otimes h_{J_2}\rangle -  \ave{b}_{J_1}\langle f ,h_{I_1}\otimes h_{J_1}\rangle \langle g, h_{I_2}\otimes h_{J_2}\rangle  \\ 
	&= ( \ave{b}_{J_2} - \ave{b}_{J_1}  )\langle f,h_{I_1}\otimes h_{J_1}\rangle\langle g, h_{I_2}\otimes h_{J_2}\rangle.
	\end{split}
\end{equation}
Recall, that we may assume the slice $b(\cdot,x_2):\R^{d_1}\to\C$ to be a constant for all $x_2\in\R^{d_2}.$

Then, similarly as in e.g. \eqref{simpleX}, we estimate  $\abs{\ave{b}_{J_2}-\ave{b}_{J_1}} \leq  \Norm{b(x_1,\cdot)}{\dot C^{0,\alpha_2}_{x_2}}\ell(V)^{\alpha_2}$ for any $x_1\in\R^{d_1}.$ Let us simply notate $\Norm{b(x_1,\cdot)}{\dot C^{0,\alpha_2}_{x_2}} = \Norm{b}{\dot C^{0,\alpha_2}_{x_2}}.$
% and notice that
%$\esssup_{x_1\in\R^{d_1}} \Norm{b(x_1,\cdot)}{\dot C^{0,\alpha_2}_{x_2}} = \Norm{b}{\dot C^{0,\alpha_2}_{x_2}}.$ 
Then, we estimate the remaining part of the commutator,
\allowdisplaybreaks \begin{align*}
&\Babs{\sum_{\substack{K\in\calD^1 \\ I_1^{i_1} = I_2^{i_2} = K }}\sum_{\substack{V\in\calD^2 \\ J_1^{j_1} = J_2^{j_2} = V }}\alpha_{(I)(J)KV}( \bave{b}_{J_2} - \ave{b}_{J_1})\langle f,h_{I_1}\otimes h_{J_1}\rangle\langle g, h_{I_2}\otimes h_{J_2}\rangle} \\
\leq& \Norm{b}{\dot C^{0,\alpha_2}_{x_2}}\int \sum_{\substack{K\in\calD^1\\ V\in\calD^2}} \ell(V)^{\alpha_2}\bave{\abs{\Delta_{K,V}^{i_1,j_1}f}}_{K\times V}\bave{\abs{\Delta_{K,V}^{i_2,j_2}g}}_{K\times V} 1_K\otimes 1_V \\
\leq& \Norm{b}{\dot C^{0,\alpha_2}_{x_2}}\BNorm{\Big( \sum_{\substack{K\in\calD^1\\ V\in\calD^2}} \big(\ell(V)^{\alpha_2}\bave{\abs{\Delta_{K,V}^{i_1,j_1}f}}_{K\times V}\big)^21_K\otimes 1_V \Big)^{1/2}}{L^{p_1}_{x_1}L^{q_2}_{x_2}} \\
&\qquad\qquad\qquad\times \BNorm{\Big( \sum_{\substack{K\in\calD^1\\ V\in\calD^2}} \bave{\abs{\Delta_{K,V}^{i_1,j_1}g}}_{K\times V}^21_K\otimes 1_V \Big)^{1/2}}{L^{p_1'}_{x_1}L^{q_2'}_{x_2}} \\
\lesssim& \Norm{b}{\dot C^{0,\alpha_2}_{x_2}}\Norm{f}{L^{p_1}_{x_1}L^{p_2}_{x_2}}\Norm{g}{L^{p_1'}_{x_1}L^{q_2'}_{x_2}},
\end{align*}
where in the last step we estimate as follows: first, for the fractional term, we note that as
\[
\Big(\ell(V)^{\alpha_2}\bave{\abs{\Delta_{K,V}^{i_1,j_1}f}}_{K\times V}\Big)^21_K\otimes 1_V \lesssim \Big(\mathsf{M}^{\alpha_2}\big( \bave{\abs{\Delta_{K,V}^{i_1,j_1}f}}_{K\times V}1_K\otimes 1_V \big)\Big)^2,
\]
by applying Lemma \ref{lem:frac:FS}, we have
\allowdisplaybreaks \begin{align*}
&\bNorm{\big( \sum_{\substack{K\in\calD^1\\ V\in\calD^2}} \big(\ell(V)^{\alpha_2}\ave{\abs{\Delta_{K,V}^{i_1,j_1}f}}_{K\times V}\big)^2 1_K\otimes 1_V\big)^{1/2}}{L^{p_1}_{x_1}L^{q_2}_{x_2}} \\
&\lesssim\bNorm{\big( \sum_{\substack{K\in\calD^1\\ V\in\calD^2}} \ave{\abs{\Delta_{K,V}^{i_1,j_1}f}}_{K\times V}^21_K\otimes 1_V \big)^{1/2}}{L^{p_1}_{x_1}L^{p_2}_{x_2}} \\ 
&\overset{*}{\lesssim}\bNorm{\big( \sum_{\substack{K\in\calD^1\\ V\in\calD^2}} \abs{\Delta_{K,V}^{i_1,j_1}f}^21_K\otimes 1_V \big)^{1/2}}{L^{p_1}_{x_1}L^{p_2}_{x_2}} \\
&\overset{**}{\leq} \bNorm{\big( \sum_{\substack{K\in\calD^1\\ V\in\calD^2}} \abs{\Delta_{K,V}^{0,0}f}^21_K\otimes 1_V \big)^{1/2}}{L^{p_1}_{x_1}L^{p_2}_{x_2}} \\ 
&\lesssim \bNorm{\mathsf{S}f}{L^{p_1}_{x_1}L^{p_2}_{x_2}} \lesssim \bNorm{f}{L^{p_1}_{x_1}L^{p_2}_{x_2}},
\end{align*}
where the $*$-estimate follows by Lemma \ref{lem:FSbp}, and the $**$-estimate follows as
\begin{align*}
	&\big(\sum_{\substack{K\in\calD^1\\ V\in\calD^2}} \abs{\Delta_{K,V}^{i_1,j_1}f}^21_K\otimes 1_V\big)^{\frac{1}{2}} = \big(\sum_{\substack{K\in\calD^1\\ V\in\calD^2}} \abs{\sum_{\substack{I^{i_1} = K \\ J^{j_1} = V}}\Delta_{I^{i_1},J^{j_1}}^{0,0}f}^21_K\otimes 1_V\big)^{\frac{1}{2}} \\
	&\leq \big(\sum_{\substack{K\in\calD^1\\ V\in\calD^2}}\sum_{\substack{I^{i_1} = K \\ J^{j_1} = V}} \abs{\Delta_{I^{i_1},J^{j_1}}^{0,0}f}^2 1_K\otimes 1_V\big)^{\frac{1}{2}} = \big(\sum_{\substack{K\in\calD^1\\V\in\calD^2}}\abs{\Delta_{K,V}^{0,0}f}^21_K\otimes 1_V\big)^{\frac{1}{2}}.
\end{align*}
The remaining non-fractional term estimates in the same fashion and we leave the details to the reader.
\end{proof}

With partial paraproducts we will use the following side of the classical and well-known $H^1$-$\BMO$ -duality.
\begin{lem}\label{lem:H1-BMO} Let $\calD$ be a dyadic grid. Then, for any arbitrary sequences $(\alpha_Q),(\beta_Q)$ there holds that 
	\begin{align*}
		\sum_{Q\in\calD}\abs{\alpha_Q}\abs{\beta_Q} \lesssim \Norm{(\alpha_Q)}{\BMO}\bNorm{\mathsf{S}_{\calD}(\beta_Q)}{L^1(\R^d)},
	\end{align*}
where,
$$\Norm{(\alpha_Q)}{\BMO} = \sup_{Q_0\in\calD}\frac{1}{\abs{Q_0}^{1/2}}\BNorm{\Big(\sum_{\substack{Q\in\calD \\ Q\subset Q_0}}\abs{\alpha_Q}^2\frac{1_Q}{\abs{Q}}\Big)^{\frac{1}{2}}}{L^2(\R^d)},\quad\mathsf{S}_{\calD}(\beta_Q) = \Big(\sum_{Q\in\calD}\abs{\beta_Q}^2\frac{1_Q}{\abs{Q}}\Big)^{\frac{1}{2}}.$$
\end{lem}

\begin{proof}[Proof of Theorem \ref{prop:c1a2}, part 2/3, partial paraproducts:] We choose the symmetry $i_1=i_2 = 0$ and consider the model operator
	\[
	 \langle S^{(0,0),j}f,g\rangle = \sum_{K\in\calD^1 }\sum_{\substack{V\in\calD^2 \\ J_1^{j_1} = J_2^{j_2} = V }}\alpha_{(J)KV}\langle f, \frac{1_K}{\abs{K}}\otimes h_{J_1}\rangle \langle g,  h_{K}\otimes h_{J_2}\rangle.
	\] 
Writing out the main term, we find out that the summand  (without the scaling factor $\alpha_{(I)(J)KV}$ in front)  in $\langle [b,S^{i,j}]f,g\rangle$ is 
	\allowdisplaybreaks \begin{align*}
	&\Big[ \big\langle f,\frac{1_K}{\abs{K}}\otimes h_{J_1}\big\rangle \big\langle bg, h_K\otimes h_{J_2}\big\rangle -  \big\langle bf,\frac{1_K}{\abs{K}}\otimes h_{J_1}\big\rangle \big\langle g, h_K\otimes h_{J_2}\big\rangle  \Big] \\
	&= \sum_{i=1,2} \big\langle f,\frac{1_K}{\abs{K}}\otimes h_{J_1}\big\rangle \big\langle A_{i}(b,g), h_K\otimes h_{J_2}\big\rangle - \sum_{i=1,2}  \big\langle A_i(b,f),\frac{1_K}{\abs{K}}\otimes h_{J_1}\big\rangle \big\langle g, h_K\otimes h_{J_2}\big\rangle \\
	&\qquad\qquad+  \Big[\big\langle f,\frac{1_K}{\abs{K}}\otimes h_{J_1}\big\rangle \big\langle A_3(b,g), h_K\otimes h_{J_2}\big\rangle -  \big\langle A_3(b,f),\frac{1_K}{\abs{K}}\otimes h_{J_1}\big\rangle \big\langle g, h_K\otimes h_{J_2}\big\rangle\Big].
	\end{align*}
The terms with the first four summands are bounded by the mixed norm estimates of bi-parameter model operators and Lemma \ref{prop:frac:para}, as in the previous case, and the difference on the last line writes out to reduce us to bounding the form
\allowdisplaybreaks \begin{align*}\label{parparform}
	\sum_{\substack{V\in\calD^2 \\ J_1^{j_1} = J_2^{j_2} = V }}\sum_{K\in\calD^1}\alpha_{(J)KV}(\ave{b}_{J_2}-\ave{b}_{J_1})\langle f, \frac{1_K}{\abs{K}}\otimes h_{J_1}\rangle \langle g,  h_{K}\otimes h_{J_2}\rangle.
\end{align*}
Then, as above, we estimate $\abs{\ave{b}_{J_2}-\ave{b}_{J_1}} \leq  \Norm{b}{\dot C^{0,\alpha_2}_{x_2}}\ell(V)^{\alpha_2}$ and this gives the desired factor $\Norm{b}{\dot C^{0,\alpha_2}_{x_2}}$ in front.
It remains to estimate as follows. By Lemma \ref{lem:H1-BMO} and the coefficient size of the partial paraproduct, we find the first estimate in the following, with the rest being straightforward or follow by lemmas \ref{lem:frac:FS} and \ref{lem:bound:sf},
\allowdisplaybreaks \begin{align*}
		&\sum_{\substack{V\in\calD^2 \\ J_1^{j_1} = J_2^{j_2} = V }}\sum_{K\in\calD^1}\abs{\alpha_{(J)KV}\ell(V)^{\alpha_2}\langle f, \frac{1_K}{\abs{K}}\otimes h_{J_1}\rangle \langle g,  h_{K}\otimes h_{J_2}\rangle} \\
		&\lesssim \sum_{\substack{V\in\calD^2 \\ J_1^{j_1} = J_2^{j_2} = V }} \frac{\abs{J_1}^{1/2}\abs{J_2}^{1/2}}{\abs{V}}\ell(V)^{\alpha_2}\bNorm{\big(\sum_{K\in\calD^1} \abs{\langle f, \frac{1_K}{\abs{K}}\otimes h_{J_1}\rangle \langle g,  h_{K}\otimes h_{J_2}\rangle}^2\frac{1_K}{\abs{K}}\big)^{\frac{1}{2}}}{L^1(\R^{d_1})} \\
		&\leq \sum_{\substack{V\in\calD^2 \\ J_1^{j_1} = J_2^{j_2} = V }} \frac{\abs{J_1}^{1/2}\abs{J_2}^{1/2}}{\abs{V}}\ell(V)^{\alpha_2}  \int_{\R^{d_1}} \mathsf{M}^1(\langle f, h_{J_1}\rangle)\mathsf{S}^1(\langle g,h_{J_2}\rangle) \\
		&= \int_{\R^{d_1}} \sum_{\substack{V\in\calD^2 \\ J_1^{j_1} = J_2^{j_2} = V }} \frac{\abs{J_1}^{1/2}\abs{J_2}^{1/2}}{\abs{V}}\ell(V)^{\alpha_2} \mathsf{M}^1\big( \int_{J_1}\Delta_V^{j_1}f h_{J_1}\big)\mathsf{S}^1\big( \int_{J_2} \Delta_V^{j_2}g h_{J_2}\big) \\
		&\leq  \int_{\R^{d_1}} \sum_{\substack{V\in\calD^2 \\ J_1^{j_1} = J_2^{j_2} = V }} \frac{\ell(V)^{\alpha_2}}{\abs{V}} \int_{J_1}  \mathsf{M}^1\Delta_V^{j_1}f\int_{J_2} \mathsf{S}^1\Delta_V^{j_2} g \\
		&= \int_{\R^{d_1}} \sum_{V\in\calD^2 } \frac{\ell(V)^{\alpha_2}}{\abs{V}} \int_{V}  \mathsf{M}^1\Delta_V^{j_1}f\int_{V} \mathsf{S}^1\Delta_V^{j_2} g \\
		&= \int_{\R^{d_1}}\int_{\R^{d_2}} \sum_{V\in\calD^2 } \ell(V)^{\alpha_2} \big\langle \mathsf{M}^1\Delta_V^{j_1}f\big\rangle_V\big\langle \mathsf{S}^1\Delta_V^{j_2} g \big\rangle_V 1_V \\
		&\leq \BNorm{ \Big( \sum_{V\in\calD^2} \big(\ell(V)^{\alpha_2} \big\langle \mathsf{M}^1\Delta_V^{j_1}f\big\rangle_V\big)^21_V \Big)^{\frac{1}{2}}}{L^{p_1}_{x_1}L^{q_2}_{x_2}}\BNorm{\Big(\sum_{V\in\calD^2}\big\langle\mathsf{S}^1\Delta_V^{j_2} g \big\rangle_V^21_V\Big)^{\frac{1}{2}}}{L^{p_1'}_{x_1}L^{q_2'}_{x_2}} \\
		&\leq  \BNorm{ \Big( \sum_{V\in\calD^2} \big(\mathsf{M}^{\alpha_2}\big(\mathsf{M}^1\Delta_V^{j_1}f1_V\big)\big)^2 \Big)^{\frac{1}{2}}}{L^{p_1}_{x_1}L^{q_2}_{x_2}}\BNorm{\Big(\sum_{V\in\calD^2}\big(\mathsf{M}^2\big(\mathsf{S}^1\Delta_V^{j_2} g1_V\big) \big)^2\Big)^{\frac{1}{2}}}{L^{p_1'}_{x_1}L^{q_2'}_{x_2}} \\
		&\lesssim  \bNorm{ \big( \sum_{V\in\calD^2} \abs{\Delta_V^{j_1}f}^2 \big)^{\frac{1}{2}}}{L^{p_1}_{x_1}L^{p_2}_{x_2}}\bNorm{\big(\sum_{V\in\calD^2}\big(\mathsf{S}^1\Delta_V^{j_2}g  \big)^2\big)^{\frac{1}{2}}}{L^{p_1'}_{x_1}L^{q_2'}_{x_2}} \\
		&\lesssim  \Norm{\mathsf{S}^2f}{L^{p_1}_{x_1}L^{p_2}_{x_2}}\Norm{\mathsf{S}g}{L^{p_1'}_{x_1}L^{q_2'}_{x_2}} \\  
		&\lesssim  \Norm{f}{L^{p_1}_{x_1}L^{p_2}_{x_2}}\Norm{g}{L^{p_1'}_{x_1}L^{q_2'}_{x_2}}.
\end{align*}
\end{proof}
\begin{proof}[Proof of Theorem \ref{prop:c1a2}, part 3/3, full paraproducts:]  Now, let $i=j=(0,0)$ and we consider the paraproduct
	\[
	\langle S^{(0,0),(0,0)}f,g\rangle = \sum_{K\in\calD^1 }\sum_{V\in\calD^2 }\alpha_{KV}\ave{f}_{K\times V} \langle g,  h_{K}\otimes h_V\rangle.
	\] 
Writing out the main term we find out that the summand (without the scaling factor $\alpha_{(I)(J)KV}$ in front) in $\langle [b,S^{i,j}]f,g\rangle$ is 
\allowdisplaybreaks \begin{align*}
&\big[\ave{f}_{K\times V} \langle bg, h_K\otimes h_V\rangle -  \ave{bf}_{K\times V}\langle g, h_K\otimes h_V \rangle  \big] \\
&= \sum_{i=1,2} \ave{f}_{K\times V} \langle A_i(b,g), h_K\otimes h_V\rangle +  \big[ \ave{f}_{K\times V}\langle A_3(b,g), h_K\otimes h_V \rangle -   \ave{bf}_{K\times V}\langle g, h_K\otimes h_V \rangle\big]. 
\end{align*}
The terms with the first two summands are bounded by the mixed norm estimates of bi-parameter model operators and Lemma \ref{prop:frac:para}, as before, and the bracketed difference on the last line writes out to reduce us to bounding the form
\allowdisplaybreaks \begin{align*}\label{fpboot}
	 \sum_{K\in\calD^1 }\sum_{V\in\calD^2 }\alpha_{KV}\bave{(\ave{b}_V-b)f}_{K\times V} \langle g,  h_{K}\otimes h_V\rangle.
\end{align*}
This is bounded by the known boundedness of the model operator and $\mathsf{M}^{\alpha_2}$ and the observation that 
\begin{align*}
	\babs{\bave{(\ave{b}_V-b)f}_{K\times V}} &\leq\bave{\abs{b-\ave{b}_V}\abs{f}}_{K\times V} \leq\Norm{b}{\dot C^{0,\alpha_2}_{x_2}(\R^{d_2})}\ave{\ell(V)^{\alpha_2}\abs{f}}_{K\times V}  \\ 
&=  \Norm{b}{\dot C^{0,\alpha_2}_{x_2}(\R^{d_2})}\ave{\ell(V)^{\alpha_2}\ave{\abs{f}}_V}_{K\times V}  \leq 	 \Norm{b}{\dot C^{0,\alpha_2}_{x_2}(\R^{d_2})}\ave{\mathsf{M}^{\alpha_2}f}_{K\times V}.
\end{align*} 

Now consider the commutator taken with the other paraproduct with the summands being $$\langle f, h_K\otimes\frac{1_V}{\abs{V}} \rangle \langle bg,  \frac{1_K}{\abs{K}}\otimes h_{V}\rangle-\langle bf, h_K\otimes\frac{1_V}{\abs{V}} \rangle \langle g,  \frac{1_K}{\abs{K}}\otimes h_{V}\rangle.$$ Again, going through with our decomposition strategy, we reduce to the operator that originates as a difference through the auxiliary operator $A_3,$
\[
\sum_{K\in\calD^1 }\sum_{V\in\calD^2 }\alpha_{KV}\big\langle (\ave{b}_V-b)f, h_K\otimes\frac{1_V}{\abs{V}} \big\rangle \langle g,  \frac{1_K}{\abs{K}}\otimes h_{V}\rangle.
\]
Again, this is bounded by the known boundedness of the model operator and the following observations, we have 
\begin{align*}
\babs{\big\langle (\ave{b}_V-b)f, h_K\otimes\frac{1_V}{\abs{V}} \big\rangle} &\leq \big\langle \abs{\ave{b}_V-b}\abs{\langle f,h_K \rangle}h_K, h_K\otimes\frac{1_V}{\abs{V}} \big\rangle \\
&\leq  \Norm{b}{\dot C^{0,\alpha_2}_{x_2}(\R^{d_2})} \big\langle \ell(V)^{\alpha_2}\abs{\langle f,h_K \rangle}h_K, h_K\otimes\frac{1_V}{\abs{V}} \big\rangle \\
&=  \Norm{b}{\dot C^{0,\alpha_2}_{x_2}(\R^{d_2})} \Big\langle \ell(V)^{\alpha_2}\big\langle\abs{\langle f,h_K \rangle}\big\rangle_Vh_K, h_K\otimes\frac{1_V}{\abs{V}} \Big\rangle\\
&=  \Norm{b}{\dot C^{0,\alpha_2}_{x_2}(\R^{d_2})} \Big\langle \sum_{L\in\calD^1}\ell(V)^{\alpha_2}\big\langle \abs{\langle f,h_L \rangle}\big\rangle_Vh_L, h_K\otimes\frac{1_V}{\abs{V}} \Big\rangle \\
&\leq  \Norm{b}{\dot C^{0,\alpha_2}_{x_2}(\R^{d_2})} \big\langle \sum_{L\in\calD^1}h_L\otimes \mathsf{M}^{\alpha_2}\langle f,h_L\rangle, h_K\otimes\frac{1_V}{\abs{V}} \big\rangle,
\end{align*}
and this time we are done as soon as we show that $\Phi f = \sum_{L\in\calD^1}h_L\otimes \mathsf{M}^{\alpha_2}\langle f,h_L\rangle$ satisfies the correct bound. For this, by duality it is enough to estimate as follows
\begin{align*}
	\abs{\langle \Phi f,g\rangle} &\leq \BNorm{\Big(\sum_{L\in\calD^1}(\mathsf{M}^{\alpha_2}\langle f,h_L\rangle)^2\frac{1_L}{\abs{L}}\Big)^{1/2}}{L^{p_1}_{x_2}L^{q_2}_{x_2}}\BNorm{\Big(\sum_{L\in\calD^1}\abs{\langle g,h_L\rangle}^2\frac{1_L}{\abs{L}}\Big)^{1/2}}{L^{p_1'}_{x_2}L^{q_2'}_{x_2}} \\
	&\lesssim \BNorm{\Big(\sum_{L\in\calD^1}\abs{\langle f,h_L\rangle}^2\frac{1_L}{\abs{L}}\Big)^{1/2}}{L^{p_1}_{x_2}L^{p_2}_{x_2}}\BNorm{\Big(\sum_{L\in\calD^1}\abs{\langle g,h_L\rangle}^2\frac{1_L}{\abs{L}}\Big)^{1/2}}{L^{p_1'}_{x_2}L^{q_2'}_{x_2}} \\
	&=  \bNorm{\mathsf{S}^1f}{L^{p_1}_{x_2}L^{p_2}_{x_2}}\bNorm{\mathsf{S}^1g}{L^{p_1'}_{x_2}L^{q_2'}_{x_2}} \lesssim \bNorm{f}{L^{p_1}_{x_2}L^{p_2}_{x_2}}\bNorm{g}{L^{p_1'}_{x_2}L^{q_2'}_{x_2}},
\end{align*}
where we again used lemmas \ref{lem:frac:FS} and \ref{lem:bound:sf}.

\end{proof}

\section{Upper bound for the case $p_1<q_1,$ $p_2=q_2$}\label{sect:last}
To treat this case, it is better to work with an alternative definition of bi-parameter CZOs.
By Grau de la Herrán \cite{Grau2016} an equivalent way to defining bi-parameter Calderón-Zygmund operators as by Martikainen \cite{Ma1} is the one by Journé \cite{Jo}. The definition of Martikainen follows quickly from Journe's and the main result in \cite{Grau2016} is the reverse direction.

\begin{defn}[Journé]\label{defn:SIO:jour} A pair $K =(K_1,K_2)$ of kernels is said to be a bi-parameter $\CZ$-kernel if for $j\in\{1,2\}$ and $i\in\{1,2\}\setminus \{j\}$ the kernels map
 $$
 K_j(x_i,y_i):\R^{d_i}\times\R^{d_i}\setminus\Delta \to \CZO(d_j,\delta),
 $$
 satisfy the bounds
 \allowdisplaybreaks \begin{align*}
 	\bNorm{K_j(x_i,y_i)}{\CZO(d_j,\delta)} \leq C\abs{x_i-y_i}^{-d_i},
 \end{align*}
 and 
 \allowdisplaybreaks \begin{align*}
 \bNorm{K_j(x_i, y_i) - K_j(x_i', y_i)}{\CZO(d_j,\delta)} +  \bNorm{K_j(y_i, x_i) - K_j(y_i, x_i')}{\CZO(d_j,\delta)}  \leq C\frac{|x_i-x_i'|^{\delta}}{|x_i-y_i|^{d_i+\delta}},
 \end{align*}
 whenever $\abs{x_i-x_i'}\leq 1/2\abs{x_i-y_i}.$ 
% The best constant in these estimates is denoted with $\Norm{K_j}{\CZO(d_j,\delta)}.$
 
 An operator $T$ with a bi-parameter $\CZ$-kernel is said to be a bi-parameter SIO if for $i\in\{1,2\}$ and $j\in \{1,2\}\setminus \{ i\}$ we have
\allowdisplaybreaks \begin{align*}
		\big\langle  T(f_1\otimes f_2),g_1\otimes g_2 \big\rangle = \int_{\R^{d_j}}\int_{\R^{d_j}}\langle K_i(x_j,y_j)f_i,g_i\rangle f_j(y_j)g_j(x_j)\ud y_j\ud x_j,
\end{align*}
 whenever $\supp(f_j)\cap\supp(g_j)=\emptyset$ and  $f_k,g_k\in \Sigma_k$ for $k\in\{i,j\}.$
 
The dual $T^{1*}$ of $T$ is given by the identity $\big\langle T^{1*} (f_1\otimes f_2),g_1\otimes g_2\big\rangle = \big\langle T(g_1\otimes f_2),f_1\otimes g_2\big\rangle.$ It is straightforward to see that $T^{1*}$ is a bi-parameter SIO if $T$ is and that the kernels of $T^{1*}$ are given by $K_1^{1*}(x_2,y_2) = K_1^*(x_2,y_2)$ and $K_2^{1*}(x_1,y_1) = K_2(x_1,y_1).$
\end{defn}
\begin{defn}A bi-parameter SIO as in Definition \ref{defn:SIO:jour} is a bi-parameter CZO if $T$ and $T^{1*}$ are bounded on $L^2(\R^d).$
%	\allowdisplaybreaks \begin{align*}
%	\Norm{T}{\CZO((d_1,d_2),\delta)} = \Norm{T}{L^2(\R^d)\to L^2(\R^d)} + \Norm{T^{1*}}{L^2(\R^d)\to L^2(\R^d)} + \sum_{j=1,2}\Norm{K_j}{\CZO(d_j,\delta)}.
%	\end{align*}
\end{defn}

The advantage with this setup is that we can now easily prove the following.
\begin{lem}\label{lem:pvrep} Let $T$ be a bi-parameter CZO. Suppose that  $b(x_1,\cdot) = constant$ and $b(\cdot,x_2)\in L^{\infty}_{\loc,x_1}.$
Then, for all $f,g\in\Sigma$ we have
	\begin{align*}
		\big\langle [b,T]f,g\big\rangle = \int_{\R^{d_1}}\int_{\R^{d_1}}(b(x_1)-b(y_1))\big\langle K_2(x_1,y_1)f(y_1,\cdot),g(x_1,\cdot)\big\rangle \ud y_1\ud x_1,
	\end{align*}
where we denote $b(x_1) = b(x_1,v)$ for some choice of $v\in\R^{d_2}.$
\end{lem}

\begin{proof} We first consider the one-parameter setting with the one-parameter space $\R^n$ and let $b\in L^{\infty}_{\loc}(\R^n).$
	 It is a basic part of the one-parameter theory (see e.g. Grafakos \cite{GrafMFA}, Proposition 4.1.11.) that for each $T\in \CZO(n,\delta)$ there exists $\wt{T}\in \CZO(n,\delta)$ and a function $m\in L^{\infty}$ so that
	\[
	(T-m)h = \wt{T}h,\qquad \widetilde{T}h(x) = \lim_{k\to \infty}\int_{\abs{x-y}>\varepsilon_k}K(x,y)h(y)\ud y
	\]
	where $K$ is the kernel of $T$ and the limit holds along some sequence $\varepsilon_k\to 0$ and for all bounded and compactly supported functions $h.$
	
	The above immediately gives the following: suppose that $b\in L^{\infty}_{\loc}(\R^n)$ and $f,g\in \Sigma_n,$ so that
	\begin{equation}\label{aid2}
		\begin{split}
			\big\langle [b,T]f,g	\big\rangle &= 	\big\langle [b,T-m]f,g	\big\rangle = 	\big\langle [b,\widetilde{T}]f,g	\big\rangle \\
		&= \int_{\R^n}\lim_{\varepsilon_k\to 0}\int_{\abs{x-y}>\varepsilon_k}(b(x)-b(y))K(x,y)f(y)g(x)\ud y\ud x \\
		&= \int_{\R^n}\int_{\R^n}(b(x)-b(y))K(x,y)f(y)g(x)\ud y\ud x.
		\end{split}
	\end{equation}
	The last step follows by the dominated convergence theorem after the following estimate uniform in $\varepsilon_k,$
	\begin{equation*}
			\int\abs{(b(x)-b(y))K(x,y)f(y)}\ud y \lesssim \Norm{b}{\dot C^{0,\alpha}(\R^d)} \int_{\R^d}\abs{x-y}^{\alpha-d}\abs{f(y)}\ud y;
	\end{equation*}
	since $f$ is bounded and compactly supported, we see that the right-hand side is finite.
	Now with this one-parameter result at hand, we turn to the claim itself. 
	
	By linearity it is enough to prove the claim for functions $f = f_1\otimes f_2$ and $g = g_1\otimes g_2$ of tensor form. If $T$ is an SIO  as by Journé, then the size estimate
	\[
	\abs{\langle K_2(x_1,y_1)f_2,g_2\rangle} \lesssim \abs{x_1-y_1}^{d_1}\Norm{f_2}{L^p}\Norm{g_2}{L^{p'}}
	\]
	is satisfied, and similarly immediately from the definitions the regularity estimates also hold.
	Consequently, since $T$ is bounded, the function 
	$(x_1,y_1)\mapsto \langle K_2(x_1,y_1)f_2,g_2\rangle$ is a kernel of the one-parameter CZO defined by
	\begin{align*}
		\langle T_{f_2,g_2}f_1,g_1\rangle = \langle T(f_1\otimes f_2),g_1\otimes g_2\rangle.
	\end{align*} 
	Then, it follows by the one-parameter result \eqref{aid2} that
	\begin{align*}
		 	\big\langle [b(\cdot,v),T_{f_2,g_2}]f_1,g_1\big\rangle = \int_{\R^{d_1}}\int_{\R^{d_1}} (b(x_1,v)-b(y_1,v))\big\langle K_2(x_1,y_1)f_2,g_2\big\rangle f_1(y_1)g_1(x_1)\ud y_1\ud x_1,
	\end{align*}
	where we note that $b(\cdot,v)\in L^{\infty}_{\loc,x_1}.$
	However, we also have
	\begin{align*}
		\big\langle [b,T](f_1\otimes f_2), g_1\otimes g_2\big\rangle &= \big\langle b(\cdot,v)T(f_1\otimes f_2) - T(b(\cdot,v)(f_1\otimes f_2),g_1\otimes g_2\big\rangle \\
		&= \big\langle T_{f_2,g_2}f_1,b(\cdot,v)g_1\big\rangle - \big\langle T_{f_2,g_2}(b(\cdot,v)f_1),g_1\big\rangle \\
		&= \big\langle [b(\cdot,v),T_{f_2,g_2}]f_1, g_1\big\rangle, 
	\end{align*}
	and thus the claim follows.
\end{proof}

\begin{prop}\label{prop:a1c2} Let $p_1<q_1$ and $p_2=q_2,$ let $T$ be a bi-parameter CZO and suppose that $b(x_1,\cdot)=constant$ and $b(\cdot,x_2) \in \dot C^{0,\alpha_1}_{x_1}.$ Then, we have 
	\[
	\Norm{[b,T]f}{L^{q_1}_{x_1}L^{p_2}_{x_2}} \lesssim 	\Norm{b(\cdot,x_2)}{\dot C^{0,\alpha}_{x_1}}\Norm{f}{L^{p_1}_{x_1}L^{p_2}_{x_2}}.
	\]
\end{prop}
\begin{proof} It is enough to prove the claim for functions in a dense subset of the space $L^{p_1}_{x_1}L^{p_2}_{x_2}$ and clearly $\Sigma$ is such a subset. As $b(\cdot,x_2) \in \dot C^{0,\alpha_1}_{x_1},$ especially $b(\cdot,x_2)\in L^{\infty}_{\loc,x_1}$ and thus by Lemma \ref{lem:pvrep} we can estimate
	\allowdisplaybreaks \begin{align*}
		\abs{	\langle [b,T]f,g\rangle} &\leq \int _{\R^{d_1}}\int_{\R^{d_1}}  \Babs{(b(x_1,v)-b(y_1,v)) \Big\langle K_{2}(x_1,y_1)f(y_1,\cdot)(z),g(x_1,z)\Big\rangle_z} \ud y_1\ud x_1 \\
		&\lesssim \int _{\R^{d_1}}\int_{\R^{d_1}}  \frac{\abs{b(x_1,v)-b(y_1,v)}}{\abs{x_1-y_1}^{d_1}} \bNorm{f(y_1,z)}{L^{p_2}_{z}}\bNorm{g(x_1,z)}{L^{p_2'}_{z}} \ud y_1\ud x_1 \\
		&\leq \bNorm{b(\cdot,v)}{\dot C^{0,\alpha_1}_{x_1}}   \int_{\R^{d_1}}\int_{\R^{d_1}} \abs{x_1-y_1}^{\alpha_1-d_1}\bNorm{f(y_1,z)}{L^{p_2}_{z}}\bNorm{g(x_1,z)}{L^{p_2'}_{z}}  \ud y_1\ud x_1 \\
		&= \bNorm{b(\cdot,v)}{\dot C^{0,\alpha_1}_{x_1}}   \int_{\R^{d_1}}  \mathsf{I}_{\alpha_1}\Big( \bNorm{f(\cdot,z)}{L^{p_2}_{z}} \Big)(x_1) \cdot \bNorm{g(x_1,z)}{L^{p_2'}_{z}} \ud x_1 \\
		&\lesssim \bNorm{b(\cdot,v)}{\dot C^{0,\alpha_1}_{x_1}}  \bNorm{f(y_1,z)}{L^{p_1}_{y_1}L^{p_2}_{z}}\bNorm{g(x_1,z)}{L^{q_1'}_{x_1} L^{p_2'}_{z}},
	\end{align*}
where in the last step we used the boundedness of the fractional integral.
\end{proof}

The author has no competing interests to declare.
\bibliography{references}

\begin{bibdiv}
\begin{biblist}

\bib{AHLMO}{article}{
      author={Airta, E.},
      author={Hyt\"{o}nen, T.},
      author={Li, K.},
      author={Martikainen, H.},
      author={Oikari, T.},
       title={{Off-Diagonal Estimates for Bi-Commutators}},
        date={202109},
        ISSN={1073-7928},
     journal={International Mathematics Research Notices},
         url={https://doi.org/10.1093/imrn/rnab239},
}

\bib{CLMS1993}{article}{
      author={Coifman, R.},
      author={Lions, P.-L.},
      author={Meyer, Y.},
      author={Semmes, S.},
       title={Compensated compactness and hardy spaces},
        date={1993},
        ISSN={0021-7824},
     journal={J. Math. Pures Appl.},
      volume={9},
       pages={247\ndash 286},
}

\bib{CRW}{article}{
      author={Coifman, R.~R.},
      author={Rochberg, R.},
      author={Weiss, G.},
       title={Factorization theorems for {H}ardy spaces in several variables},
        date={1976},
     journal={Ann. of Math. (2)},
      volume={103},
      number={3},
       pages={611\ndash 635},
}

\bib{CMN2019}{article}{
      author={D.~Cruz-Uribe, H.~Nguyen, K.~Moen},
       title={Multilinear fractional calder\'{o}n-zygmund operators on weighted
  hardy spaces},
     journal={preprint},
      eprint={https://arxiv.org/abs/1903.01593},
}

\bib{GrafMFA}{book}{
      author={Grafakos, L.},
       title={Modern fourier analysis},
    language={eng},
     edition={3rd ed. 2014.},
      series={Graduate Texts in Mathematics, 250},
   publisher={Springer New York},
     address={New York, NY},
        ISBN={1-4939-1230-5},
}

\bib{Grau2016}{article}{
      author={Grau de~la Herr\'{a}n, A.},
       title={Comparison of {$T1$} conditions for multi-parameter operators},
        date={2016},
     journal={Proc. Amer. Math. Soc.},
      volume={144},
      number={6},
       pages={2437\ndash 2443},
}

\bib{HPW2018}{article}{
      author={Holmes, I.},
      author={Petermichl, S.},
      author={Wick, B.~D.},
       title={Weighted little bmo and two-weight inequalities for {J}ourn\'e
  commutators},
        date={2018},
     journal={Anal. PDE},
      volume={11},
      number={7},
       pages={1693\ndash 1740},
}

\bib{HyLpLq}{article}{
      author={Hyt\"{o}nen, T.},
       title={The ${L}^p$-to-${L}^q$ boundedness of commutators with
  applications to the {J}acobian operator},
        date={2021},
        ISSN={0021-7824},
     journal={Journal de Math\'{e}matiques Pures et Appliqu\'{e}es},
      volume={156},
       pages={351\ndash 391},
}

\bib{HMV}{article}{
      author={Hyt\"onen, T.},
      author={Martikainen, H.},
      author={Vuorinen, E.},
       title={Multi-parameter estimates via operator-valued shifts},
        date={2019},
     journal={Proc. Lond. Math. Soc.},
      volume={119},
      number={6},
       pages={1560\ndash 1597},
}

\bib{Iwa1997}{article}{
      author={Iwaniec, T.},
       title={Nonlinear commutators and jacobians},
     journal={J. Fourier Anal. Appl. 3 (1997), Special Issue, 775--796},
}

\bib{Janson1978}{article}{
      author={Janson, S.},
       title={Mean oscillation and commutators of singular integral operators},
        date={1978},
     journal={Ark. Mat.},
      volume={16},
       pages={263\ndash 270},
}

\bib{Jo}{article}{
      author={Journ\'{e}, J.-L.},
       title={{C}alder\'on-{Z}ygmund operators on product spaces},
        date={1985},
     journal={Rev. Mat. Iberoam.},
      volume={1},
       pages={55\ndash 91},
}

\bib{kuffner2018weak}{article}{
      author={Kuffner, M.-J.},
       title={Weak factorization of the hardy space hp for small values of p,
  in the multilinear setting},
        date={2020},
        ISSN={0022-247X},
     journal={Journal of Mathematical Analysis and Applications},
      volume={485},
      number={1},
       pages={123711},
}

\bib{LMV2019biparcom}{article}{
      author={Li, K.},
      author={Martikainen, H.},
      author={Vuorinen, E.},
       title={{Bloom-Type Inequality for Bi-Parameter Singular Integrals:
  Efficient Proof and Iterated Commutators}},
        date={201904},
        ISSN={1073-7928},
     journal={International Mathematics Research Notices},
      volume={2021},
      number={11},
       pages={8153\ndash 8187},
         url={https://doi.org/10.1093/imrn/rnz072},
}

\bib{LMV:Bloom}{article}{
      author={Li, K.},
      author={Martikainen, H.},
      author={Vuorinen, E.},
       title={{Bloom-Type Inequality for Bi-Parameter Singular Integrals:
  Efficient Proof and Iterated Commutators}},
        date={2019},
        ISSN={1073-7928},
     journal={Int. Math. Res. Not. IMRN},
}

\bib{Lindberg2017}{article}{
      author={Lindberg, S.},
       title={On the hardy space theory of compensated compactness quantities},
        date={2017},
        ISSN={0003-9527},
     journal={Archive for Rational Mechanics and Analysis},
      volume={224},
       pages={709\ndash 742},
}

\bib{Ma1}{article}{
      author={Martikainen, H.},
       title={Representation of bi-parameter singular integrals by dyadic
  operators},
        date={2012},
     journal={Adv. Math.},
      volume={229},
      number={3},
       pages={1734 \ndash  1761},
}

\bib{Nehari1957}{article}{
      author={Nehari, Z.},
       title={On bounded bilinear forms},
        date={1957},
     journal={Annals of Mathematics},
      volume={65},
       pages={153\ndash 162},
}

\bib{Uchiyama1978}{article}{
      author={Uchiyama, A.},
       title={On the compactness of operators of hankel type},
        date={1978},
     journal={T\^ohoku Math},
      volume={30(1)},
       pages={163\ndash 171},
}

\end{biblist}
\end{bibdiv}

\end{document}